\documentclass[11pt,reqno]{amsart}
\usepackage{xifthen,setspace,bm}

\usepackage[utf8]{inputenc}
\usepackage[T1]{fontenc} 
\usepackage{bm}
\usepackage{amsmath,amsthm,amsfonts,amssymb,mathrsfs,stmaryrd,bigints,graphicx}
\usepackage{dsfont}
\usepackage{mathtools}
\usepackage[usenames]{color}
\usepackage{thmtools, thm-restate}
\usepackage[normalem]{ulem}

\usepackage[colorlinks=true,linkcolor=blue]{hyperref}
\usepackage{epstopdf} 
\usepackage{amssymb}
\usepackage{setspace}
\usepackage{enumerate}
\usepackage{bigstrut}
\usepackage{multirow}
\usepackage{mathtools,xparse}
\usepackage{mathrsfs}
\usepackage{hyperref}
\usepackage{xcolor}
\newtheorem*{lemma*}{Lemma}
\newtheorem{theorem}{Theorem}[section]

\newtheorem{lemma}[theorem]{Lemma}
\newtheorem{proposition}[theorem]{Proposition}

\theoremstyle{definition}
\newtheorem{definition}{Definition}
\newtheorem{remark}[theorem]{Remark}

\allowdisplaybreaks
\usepackage{chngcntr}

\renewcommand{\P}{\mathbb{P}}

\newcommand{\E}{{\mathbb{E}}}
\newcommand{\1}{\mathds{1}}

\newcommand{\R}{\mathbb{R}}
\newcommand{\norm}[1]{\left\lVert#1\right\rVert}

\newcommand{\ep}{\varepsilon} 
\newcommand{\e}{\varepsilon}

\newcommand{\N}{\mathbb{N}}

	\renewcommand{\P}{\mathbb{P}}

\newcommand{\cE}{\mathcal{E}}
\newcommand{\cF}{\mathcal{F}}
\newcommand{\cG}{\mathcal{G}}

\newcommand{\cM}{\mathcal{M}}

\renewcommand{\setminus}{\backslash}

\DeclareMathOperator{\dist}{dist}
\newtheorem{maintheorem}{Theorem}

\def\ba{\begin{align}}
\def\ea{\end{align}}
\def\bs{\begin{split}}
\def\es{\end{split}}

\newcommand{\paren}[1]{\left( #1 \right)}
\newcommand{\parenn}[1]{\left\{ #1 \right\}}

\newcommand{\size}[1]{\left| #1 \right|}

\begin{document}

\title{Extremal spectral behavior of weighted random $d$-regular graphs}
\author{Jaehun Lee and Kyeongsik Nam}

\begin{abstract}
Analyzing the spectral behavior of random matrices with \emph{dependency}  among entries is a challenging problem. The adjacency matrix of the random $d$-regular graph is a prominent example that has attracted immense interest. A crucial spectral observable  is the extremal eigenvalue, which reveals useful geometric properties of the graph.  According to the Alon's conjecture, which was verified by Friedman \cite{Friedman08}, the (nontrivial) extremal eigenvalue of the random $d$-regular graph   is approximately $2\sqrt{d-1}$.

    In the present paper,
we analyze the extremal spectrum  of  the random $d$-regular graph  (with $d\ge 3$ fixed)  equipped with random edge-weights,  and  precisely describe its phase transition behavior with respect to the tail of edge-weights.  
 In addition, we establish that  the extremal eigenvector is always localized, showing a sharp contrast to the unweighted case  where all eigenvectors are delocalized. 
 Our method is robust and inspired by a  sparsification technique developed in the context of Erd\H{o}s-R\'{e}nyi graphs \cite{GN22}, 
 which can also be applied to analyze  the spectrum of general  random matrices  whose entries are dependent.
\end{abstract}

\address{Department of Mathematics, HKUST, Hong Kong} 
\email{jaehun.lee@ust.hk}

\address{Department of Mathematical Sciences, KAIST, South Korea}
\email{ksnam@kaist.ac.kr}

\maketitle

\tableofcontents

\section{Introduction}
 
Spectral statistics arising from random graphs and networks has been the subject of significant importance. Erd\H{o}s-R\'{e}nyi graph  $\cG_{n,p}$ is one of crucial types of networks, where each edge is included in the graph of size $n$ with probability  $p$, independently from every other edge. Another  fundamental model is the random $d$-regular graph $\cM_{n,d}$,  a uniform distribution on  the collection of (simple) $d$-regular graphs of size $n$. A crucial feature of the random $d$-regular graph is a presence of  \emph{dependency} between edges.

For the purpose of 
encoding  physically relevant phenomena, edges 
in the network are  
often 
equipped with random weights, which 
denote resistances 
in the 
context of 
electrical networks. For instance, the adjacency matrix 
of the Erd\H{o}s-R\'{e}nyi graph $\cG_{n,p}$ 
equipped with 
random edge-weights can  be 
viewed as a  {sparse} version of Wigner matrices.

Spectral properties of the weighted Erd\H{o}s-R\'{e}nyi graph $\cG_{n,p}$ have been extensively studied and are quite well-understood. A particularly central question is to verify the  {universality} principle which refers to the phenomenon that the asymptotic spectral behavior does not depend on the type of entry distributions. For instance, if the sparsity $p$ satisfies $p\gg \frac{1}{n}$, then  the appropriately scaled empirical spectral distribution  converges to the Wigner’s semicircle law \cite{rodgers1988density, KKPS89}, and the re-scaled largest eigenvalue lies near the edge of the  semicircle law \cite{MR1825319, MR3098073, MR2964770}. In contrast, in the case of constant average degree of sparsity (i.e.~$p=\frac{d}{n}$ with $d$ fixed), this universal behavior  breaks down and the spectrum heavily depends on the entry distribution. Although it is currently intractable to examine entire eigenvalues of such sparse Wigner matrices, the largest eigenvalue was successfully analyzed in the recent work \cite{GN22, GHN22+}.

On the other hand, it is much more delicate to analyze the spectral behavior of random matrices whose entries are \emph{not} independent. The  prominent example includes the adjacency matrix of the random $d$-regular graph  which arises naturally in  theoretical
computer science, combinatorics
and statistical 
physics \cite{Anantharaman17, cohen2016ramanujan, MR963118, MR3374962, MR3892446}.
Spectral information of the random $d$-regular graph reveals crucial geometric features of the graph. For instance, the spectral gap, a difference between the first  and second largest eigenvalues, measures the expanding property of the graph. 
The celebrated Alon's conjecture \cite{Alon86} states that random $d$-regular graphs (with $d\ge 3$ fixed) are weakly 
Ramanujan with high probability\footnote{We say that a series of events $\{E_n\}_{n\ge 1}$ occur \emph{with high probability} if $\lim_{n\to\infty}\P(E_n)=1$. }, i.e.~all nontrivial eigenvalues are {essentially} bounded by $2\sqrt{d-1}$ in absolute value. This was first proved by 
Friedman \cite{Friedman08} and subsequently by  Bordenave \cite{Bordenave20}, using sophisticated moment methods. Recently, it is shown \cite{HY21+} that (nontrivial) extremal eigenvalues are 
concentrated  around $2\sqrt{d-1}$ with a polynomial 
error bound, with the aid of resolvent methods.

However, once the entries of the adjacency matrix of random $d$-regular graphs are allowed to take general values other than 0 or 1,
the spectral information has not been known so far.  
Such random networks naturally arises in the context of the neural network theory \cite{sompolinsky1988chaos, MR3403052}. 
 The goal of the present paper is to explore the  spectral behavior of general random networks  on the random $d$-regular graph, i.e.~\emph{weighted} random $d$-regular graph. We
  precisely describe its extreme eigenvalue in terms of the tail distribution of edge-weights. In addition, we analyze the localization phenomenon of the extreme eigenvector. Our methodology is robust and can also be applied to examine the spectral behavior of a wide class of random matrices whose entries are dependent.

Before presenting our main results, we 
emphasize some of 
the key features of our 
results. The most fundamental 
one is the law of large 
numbers (LLN) result for the largest eigenvalue of the weighted random $d$-regular graph. We  show that the extremal eigenvalue exhibits a (continuous) phase transition with respect to the   edge-weight distributions, and verify that the transition occurs when the edge-weights possess a Gaussian tail. In addition, we prove that  $\ell^2$-mass of the extremal eigenvector  is essentially supported only on $n^\e$ 
 vertices  ($\e>0$ is arbitrary constant)  with high probability. This shows  a substantial difference from the case of \emph{unweighted} random $d$-regular graphs, where all eigenvectors are shown to be delocalized \cite{HY21+}.

~

Let us now precisely define our model and state our main results.

\subsection{Main results}
Throughout the paper, for any graph $G$, we denote by $V(G)$ and   $E(G)$  the set of vertices and   edges in $G$ respectively. Each element in $E(G)$ will be written as  $(i,j) =  (j,i)$ with $i,j\in V(G)$.

We first recall the definition of the  random $d$-regular graph.
\begin{definition}[Random $d$-regular graph]
    The \emph{random $d$-regular graph} is the random graph chosen uniformly from the set of (simple) $d$-regular graphs with a vertex set $[n]\coloneqq \{1,2,\cdots,n \}$.
\end{definition}
We denote by $A = (A_{ij})_{i,j\in [n]}$ the adjacency matrix of the random $d$-regular graph. In other words, $A$ is the matrix ensemble uniformly chosen among all symmetric matrices of size $n\times n$, whose entries are either 0 or 1, such that
\begin{align*}
    A_{ii}=0, \quad \textstyle{\sum}_{j=1}^{n}A_{ij}=d, \quad \forall i\in [n].
\end{align*}
Throughout this paper, we shall identify a graph with its adjacency matrix for brevity.

~

Now, we introduce the matrix ensemble considered in this paper.
Let $W = (W_{ij })_{i,j\in [n] }$ be a standard (symmetric) Wigner matrix independent of $A$, in other words $W_{ij }=W_{ji}$ 
and $\{W_{ij}\}_{1\le i< j\le n}$ are i.i.d random
variables. The matrix of interest  is {the Hadamard product} $X =A  \odot W$, i.e., $X_{ij} =A_{ij}W_{ij}$. This is a sparse random
matrix which can be regarded as the adjacency matrix of a network, whose underlying graph is the random $d$-regular graph $A$ induced with i.i.d.~edge-weights $\{W_{ij}\}_{1\le i< j\le n}$.  We assume that each edge-weight $W_{ij}$ is given by  the following Weibull random variables with a shape parameter $\alpha>0$:\begin{definition}[Weibull distribution]\label{def: weights}
 $W$ is called a \emph{Weibull} random variable with a shape parameter $\alpha>0$ if  
 there exist constants $C_{1}, C_{2}>0$ such that for all $t \ge 1$,
		\begin{equation*}
		C_{1}\exp(-t^{\alpha})/2 \le \P(W \ge t) \le C_{2}\exp(-t^{\alpha})/2,
		\end{equation*}
		and
		\begin{equation*}
		C_{1}\exp(-t^{\alpha})/2 \le \P(W \le -t) \le C_{2}\exp(-t^{\alpha})/2.
		\end{equation*}
	In particular, for all $t \ge 1,$
		\begin{equation*}
		C_{1}\exp(-t^{\alpha}) \le \P(|W| \ge t) \le C_{2}\exp(-t^{\alpha}).
		\end{equation*}
\end{definition}

We denote the eigenvalues of $X$, in a non-increasing order, by
\begin{equation*}
\lambda_{1}(X)\ge\lambda_{2}(X)\ge\cdots\ge\lambda_{n}(X).
\end{equation*}
One of the main results is about the asymptotic behavior of the largest eigenvalue, $\lambda_{1}(X)$. For any (fixed) integer $d\ge 3$, let us denote by $\mathbb{T}_d$  the infinite  $d$-regular tree.
\begin{maintheorem}[Law of large numbers for the largest eigenvalue]\label{thm: LLN} 
 Let $d\geq 3$ be any (fixed) integer. Then, there exists a continuous function $h_d:(2,\infty) \rightarrow (1,\infty)$ such that
	\begin{equation*}
	\lim_{n\to\infty} \frac{\lambda_{1}(X)}{(\log n)^{\frac{1}{\alpha}}} =
	\begin{cases*}
	1 & $0<\alpha\le 2$,\\
	 h_d(\alpha)& $\alpha > 2$,
	\end{cases*}
	\end{equation*}
	in probability. The function $h_d(\alpha)$ is given by
 \begin{align} \label{h}
     h_d(\alpha) \coloneqq  2\cdot \sup_{\textup{\textbf{u}}=(u_i)_{i \in V(\mathbb{T}_d)}, \norm{\textup{\textbf{u}}}_1=1} \Big( \sum_{ (i,j)\in E
     (\mathbb{T}_{d}) } 
 |u_{i}u_j| ^{ \frac{\alpha}{2(\alpha-1)} } \Big)^{ \frac{\alpha-1}{\alpha} },\quad \alpha>2,
 \end{align}
which satisfies $\lim_{\alpha \downarrow 2}h_d(\alpha) = 1 $ and  $\lim_{\alpha \rightarrow \infty} h_d(\alpha)  =2 \sqrt{d-1} .$
\end{maintheorem}

Theorem \ref{thm: LLN}  first establishes the law of large numbers (LLN) behavior for the extreme eigenvalue of weighted random $d$-regular graphs. Authors might notice the following  aspects of Theorem  \ref{thm: LLN}:
\begin{enumerate}
    \item Limiting behavior of   the (normalized) largest eigenvalue, as $\alpha\to \infty$, is equal to $2\sqrt{d-1}$. This quantity is  the
    typical value  of the (nontrivial) largest eigenvalue of the \emph{unweighted} random $d$-regular graph, which  formally corresponds to the case ``$\alpha=\infty$''.
    \item Since $h_d(\alpha)>1$ for $\alpha>2$, the largest eigenvalue of the weighted random $d$-regular graph exhibits a phase transition at $\alpha=2,$ i.e.~when the edge-weights possess a Gaussian tail.
\end{enumerate}

We elaborate on these aspects in more details in the following remarks.
\begin{remark}[\emph{Unweighted} random $d$-regular graphs: Limiting behavior as $\alpha\rightarrow \infty$]
The random $d$-regular graph \emph{without} edge-weights can be formally regarded as the case ``$\alpha=\infty$''.  Note that by Alon's conjecture \cite{Alon86}, which was first verified by Friedman  \cite{Friedman08} and further by Bordenave \cite{Bordenave20}, the (nontrivial) largest eigenvalue of the random $d$-regular graph is concentrated near $2\sqrt{d-1}$ with high probability. Hence, this is compatible with the result $\lim_{\alpha \rightarrow \infty} h_d(\alpha)  =2 \sqrt{d-1}$ in  Theorem \ref{thm: LLN}.

Asymptotic behaviors of the (nontrivial) largest eigenvalue of unweighted random $d$-regular graphs is obtained using a sophisticated moment method.  However, the moment method is not applicable to analyze the extremal spectrum of \emph{weighted} random $d$-regular graphs.
As described in Section \ref{idea} later, our methodology is completely different from the moment \cite{Friedman08, Bordenave20} and resolvent method \cite{HY21+}, and rather  is based on the idea of sparsification exploited in the context of Erd\H{o}s-R\'{e}nyi graphs \cite{GN22}. It is an interesting problem to reprove the Alon's conjecture using our methods, by inducing edge-weights with  an  appropriate shape parameter $\alpha$ chosen to be slowly growing in $n$.
\end{remark}

\begin{remark}[\emph{Continuous} phase transition at $\alpha=2$]
    Theorem \ref{thm: LLN}, in particular $h_d(\alpha)>1$ for $\alpha>2$, says that the largest eigenvalue of the weighted random $d$-regular graph exhibits a phase transition at $\alpha=2$, i.e.~when the edge-weights possess a Gaussian tail. Since  $\lim_{\alpha \downarrow 2}h_d(\alpha) = 1 $, the phase transition occurs continuously.
\end{remark}

\begin{remark}[Spectral norm]
    By considering the matrix $-X = A \odot (-W)$, one can deduce that the largest eigenvalue of $-X$, which is equal to $-\lambda_n(X)$, satisfies the result in Theorem \ref{thm: LLN} as well. Therefore, we establish that the spectral norm  $\norm{X}$ satisfies
    \begin{equation*}
	\lim_{n\to\infty} \frac{\norm{X}}{(\log n)^{\frac{1}{\alpha}}} =
	\begin{cases*}
	1 & $0<\alpha\le 2$,\\
	 h_d(\alpha)& $\alpha > 2$,
	\end{cases*}
	\end{equation*}
	in probability.
\end{remark}

\begin{remark}[Random $d$-regular graph versus Erd\H{o}s-R\'{e}nyi graph]\label{compare}
Random $d$-regular graphs $\mathcal{M}_{n,d}$ and Erd\H{o}s-R\'{e}nyi graphs $\cG_{n,d/n}$ possess notably different  structures when $d$ is \emph{fixed}. For instance, in the latter case, there exists a vertex whose degree is of order $\frac{\log n}{\log \log n}$.
Therefore,  it is natural to expect that
    extreme eigenvalues of $\mathcal{M}_{n,d}$ and $\cG_{n,d/n}$ exhibit  different behaviors as well. Indeed, in the recent work \cite{GHN22+}, it is shown that the largest eigenvalue $\lambda_1$ of the weighted  Erd\H{o}s-R\'{e}nyi graph $\cG_{n,d/n}$ (whose edge-weights are given by i.i.d.~Weibull distributions with a shape parameter $\alpha>0$) satisfies
  \begin{align} \label{er}
     \lambda_1 \approx 
     \begin{cases*}
         (\log n)^{1/\alpha}&\quad $0<\alpha \le  2$,\\
         C_\alpha\frac{(\log n)^{1/2}}{(\log \log n)^{1/2 - 1/\alpha}}&\quad  $\alpha > 2$
     \end{cases*}
  \end{align}
($C_\alpha>0$ is an explicit constant). This together with Theorem \ref{thm: LLN} imply that when the edge-weights possess heavy-tails (i.e.~$\alpha\leq 2$), the typical value of the extremal eigenvalue of weighted random $d$-regular graphs  is same as that of  weighted  Erd\H{o}s-R\'{e}nyi graphs. Indeed in both cases, the largest eigenvalue is governed by the largest value of edge-weights.

However, in the light-tail case
$\alpha>2$, by Theorem \ref{thm: LLN} and \eqref{er}, the largest eigenvalue $\lambda_{1}$ of weighted   Erd\H{o}s-R\'{e}nyi graphs is much larger than that of weighted random $d$-regular graphs. In fact, governing  spectral mechanisms are substantially different. In the case of weighted random $d$-regular graphs, $\lambda_1$ is determined by the ``almost'' infinite $d$-regular tree induced with high edge-weights (see the discussion after Theorem \ref{thm: localization}).
Whereas  in the case of weighted  Erd\H{o}s-R\'{e}nyi graphs,   $\lambda_1$ is determined by a ``star graph'' of size  $\Theta(\frac{\log n}{\log \log n})$   with high 
edge-weights \cite{GHN22+}. This structural difference occurs since the Erd\H{o}s-R\'{e}nyi graph can have atypically large degree vertex whose associated star graph yields the desired value of $\lambda_1$. On the other hand, in the case of random $d$-regular graphs, since all degrees are equal to $d$, a   $d$-regular tree of a   large enough depth  is needed to attain the desired value of $\lambda_1$.
\end{remark}

Given the understanding on the largest eigenvalue, the next  fundamental question is to analyze the corresponding  top eigenvector. 
As discrete analogues 
of compact negatively 
curved surfaces, 
 eigenvectors of the 
 random $d$-regular graph have been extensively 
 studied in the context of quantum
 ergodicity (see \cite{icm2018} for the references). It is known that eigenvectors of (deterministic) $d$-regular graphs, possessing a locally tree-like structure, exhibit a \emph{delocalization}  property, i.e.~$\ell^2$-mass of eigenvectors cannot  be concentrated on the relatively small number of vertices \cite{MR3038543}. 
 Very recently, all eigenvectors of random $d$-regular graphs ($d\ge 3$ fixed) are shown to be completely delocalized with high probability \cite{HY21+}.

In this paper, we further pursue this direction and analyze the localization and delocalization property of the top eigenvector of the weighted random $d$-regular graph. The second main result of our paper is that the top eigenvector exhibits a \emph{localization}, once edge-weights are induced on the random $d$-regular graph. 
\begin{maintheorem}[Localization of the top eigenvector]\label{thm: localization} 
Let $d\geq 3$ be any (fixed) integer and $\alpha>0$. Let $\textup{\textbf{f}}=(f_i)_{i\in [n]}$ be a top eigenvector of $X$. Then, for any  $\e>0$,  for sufficiently large $n$ (depending on $\alpha$ and $\e$),
{with probability at least  $1-e^{ -   (\log n)^{1/2} }$},
there exists a subset $\mathcal{I} \in [n]$ with $ |\mathcal{I}| \le n^{\e }$  such that
    \begin{align} \label{2}
        \Big( \sum_{i\in \mathcal{I}} |f_i|^2 \Big)^{1/2}\ge (1-\e)  \norm{\textup{\textbf{f}}}_2.
    \end{align}
If in particular $\alpha \in (0,2)$, then there is a constant $\zeta_\alpha \in (0,1)$ such that {with probability at least $1-e^{ -   (\log n)^{1/2} }$}, there exists a collection $\mathcal{J}$ of at most $e^{\e(\log n)^{\zeta_{\alpha}}}$  vertex-disjoint edges in the random $d$-regular graph $A$ such that \eqref{2} holds with $\mathcal{I} \coloneqq \{v\in [n]: \text{$v$ is a vertex of some edge in $\mathcal{J}$}\}$.
     
\end{maintheorem}

    Recalling that all eigenvectors, including edge eigenvectors, of  the ``unweighted''
 random $d$-regular graph (i.e.~formally, ``$\alpha= \infty$'') are delocalized \cite{HY21+},
Theorem \ref{thm: localization} shows that the behavior of the top eigenvector of the weighted random $d$-regular graph is substantially different from the unweighted case.  
In particular, the second statement   claims that when $0<\alpha<2,$ the top eigenvector is essentially supported on the vertices of sub-polynomially  many vertex-disjoint edges. Furthermore when $\alpha\ge 2$, our proof  shows that  the top eigenvector is supported on  at most $n^{\e}$ vertex-disjoint $d$-regular trees of reasonably controlled size.  
We refer to Remark \ref{local} for the elaboration on this.
     
\begin{remark}[More general weights]\label{poly}
Although we work with edge-weight distributions satisfying Definition \ref{def: weights}  for the sake of brevity, one can  generalize  this condition as follows. By a  simple rescaling argument with  
a little additional work, one can still establish Theorems \ref{thm: LLN} and \ref{thm: localization} for edge-weight distributions satisfying  that for $t\geq 1,$
\begin{align*}
\frac{C_1}{2}   t^{-c_1} e^{-\eta t^\alpha}  \leq \P\big (W   \geq  t \big ) \leq   \frac{C_2}{2} t^{-c_2} e^{-\eta t^\alpha}  \quad  \text{ and } \quad     \frac{C_1}{2}   t^{-c_1} e^{-\eta t^\alpha}  \leq \P \big ( W \leq  -t \big ) \leq  \frac{C_2}{2} t^{-c_2} e^{-\eta t^\alpha}
\end{align*}
($\eta>0$ is a  scale parameter and $c_1,c_2\geq 0$ are constants). Note that this includes the Gaussian distribution which corresponds to the case $\alpha = 2$  and $\eta = c_1=c_2 = \frac{1}{2}$.

\end{remark}

\subsection{Related results}
We briefly review previously known results regarding the spectral behavior of random graphs and matrices, by focusing on  (weighted) Erd\H{o}s-R\'{e}nyi graphs  and random $d$-regular graphs.

\subsubsection{Erd\H{o}s-R\'{e}nyi 
 graph}
Erd\H{o}s-R\'{e}nyi graph $\cG_{n,p}$ is a fundamental model of random graphs, where  every edge is included independently of the other edges. Its spectral statistics has been extensively studied so far \cite{MR2964770, MR3098073, MR4515695, MR3800840, MR4288336, MR4021251, MR4089498}. As mentioned before, in a dense regime $p\gg \frac{1}{n}$, the macroscopic spectral behavior of  $\cG_{n,p}$ resembles  that of standard Wigner matrices. However, this  ``dense'' phenomenon completely breaks down in the regime of {constant} average degree (i.e.~$p=\frac{d}{n}$ with $d$ 
 \emph{fixed}). In this sparsity regime, the spectral behavior is heavily affected by the geometry of the graph. For instance, the largest eigenvalue is determined by a vertex of maximum degree 
(see \cite{MR1967486} for details).
 
Spectrum of the Erd\H{o}s-R\'{e}nyi  graph induced with random edge-weights, or equivalently sparse Wigner matrices, has also attracted an immense interest. In the case of \emph{weighted} $\cG_{n,p}$ with dense $p$, 
the  spectral density
converges to the semicircle law \cite{rodgers1988density, KKPS89} and
the largest eigenvalue lies near the edge of the support of the semicircle distribution \cite{MR1825319}.
We also refer to the recent works \cite{bbk_dense, tikhomirov20} where some of the  assumptions in \cite{MR1825319} are relaxed.
Furthermore, the edge universality was established   in a dense regime of sparsity \cite{MR3098073, MR2964770, MR3800840, MR4089498, MR4288336}.

However, these results are  only valid when $p \gg \frac{1}{n}$ and break down when $p$ is proportional to $ \frac{1}{n} $. In this case, the classical arguments are not applicable and completely new ideas are needed to analyze the spectral behavior.
Very recently, the largest eigenvalue of such sparse Wigner matrices  is successfully analyzed \cite{GN22, GHN22+} (see \eqref{er} for details). 
Indeed, the largest eigenvalue is determined by the intricate interplay between  the geometry of the underlying Erd\H{o}s-R\'{e}nyi graph and  the value of edge-weights.  

\subsubsection{Random $d$-regular graph} 
The random $d$-regular graph is another fundamental model of random graphs. However,
due to dependency between edges, it is much more delicate to analyze its spectral properties.  The celebrated Kesten–McKay law describes the limiting (macroscopic) behavior of bulk eigenvalues.  
Regarding the edge spectrum, 
the second largest eigenvalue and the spectral gap have a significant importance due to its wide applications in graph theory and computer science \cite{broder1987second, MR2869010, bilu2006lifts}. 
By a series of important works \cite{Alon86, MR1124768, MR1208809, Friedman08, Bordenave20}, random $d$-regular graphs (with \emph{fixed} $d\ge 3$) are known to  have the second largest eigenvalue close to $2\sqrt{d-1}$ with high probability. A finer behavior around this value, conjectured to exhibit the Tracy-Widom distribution, is still open.  
{Very recently, the conjectured Tracy-Widom fluctuation is verified for certain dense regimes, i.e.~$d\rightarrow \infty$ as $n\to \infty$ \cite{he2022spectral, HY23+}.}

In the context of eigenvectors, there has been 
extensive interest in 
verifying the 
delocalization of eigenvectors 
of random $d$-regular graphs, in the context of quantum ergodicity \cite{MR3038543}.  
Recently,  for any fixed  $d \geq 3$, all eigenvectors  are shown to exhibit a complete delocalization  \cite{BHY19}.
It is crucial to note that this delocalization  result breaks down for the  Erd\H{o}s-R\'{e}nyi graph  in the regime of constant average degree, since in this case the largest eigenvalue is governed by a subgraph induced by an atypically large degree vertex and thus expected to exhibit a localization. 

As in the case of  Erd\H{o}s-R\'{e}nyi graphs,  a \emph{weighted} version of the  random $d$-regular graph is a natural model of consideration. This can be used as a mathematical model of synaptic matrices in neuroscience \cite{sompolinsky1988chaos, rajan2006eigenvalue, MR3403052}.
However, once edge-weights are induced, to the best of authors' knowledge, only few spectral results are known \cite{rrgweighted,bethe,MR3686490}.
In  \cite{rrgweighted}, for fixed $d$, authors showed the existence of the limiting spectral density for a large class of edge-weight distributions (see also \cite{bethe} for the Lifshitz tail behavior  on the Bethe lattice with a {bounded} disorder). Also in \cite{MR3686490}, the spectral density of weighted random $d$-regular ``directed'' graphs is shown to be asymptotically the circular law in a ``dense'' regime. Regarding the edge spectrum, no information was known so far.

\subsection{Notations}
Throughout the paper, the
letter $C$ denote a positive constant, whose values may change from line to line in the proofs.
Symbols $O(\cdot)$ and $o(\cdot)$  denote the standard big-$O$ and little-$o$ notation. For nonnegative ($n$-dependent) quantities $f_{n}$ and $g_{n}$, we write $f_{n}\lesssim g_{n}$ if there exists a constant $C>0$ such that $f_{n}\le Cg_{n}$ for all $n\in \N$. We say $f_{n}\asymp g_{n}$ if  $f_{n}\lesssim g_{n}$ and $g_{n}\lesssim f_{n}$.
{For $n\in\mathbb{N}\coloneqq\{1,2,\cdots\}$, we use the notation $[n]\coloneqq\{1,2,\cdots,n\}$ for brevity.}
Finally we denote by $G = (V,E,Y)$ a network having an underlying graph $G = (V,E)$ (we abuse the notation) with $Y$ as  a conductance matrix.

\subsection{Idea of proof}\label{idea}
$\empty$\\
\emph{Lower bound on the largest eigenvalue}
We recall the classical fact that random $d$-regular graphs possess a  locally tree-like structure, i.e., $\Theta(\log n)$-neighborhood  of the most of vertices is a $d$-regular tree.  We  lower bound the largest eigenvalue by inducing high value edge-weights  on these  $d$-regular  trees that are vertex-disjoint. More precisely, for each such $d$-regular tree $T$, we take edge-weights in an optimal way so that $\lambda_1(T)$ becomes large with the smallest probability cost. This  probability cost is related to the variation problem \eqref{h}.

~

\noindent
\emph{Upper bound on the largest eigenvalue}
At the high level, we analyze two parts of the network decomposed according to the value of edge-weights. 
Since the random $d$-regular graph has a uniformly bounded spectral norm, the network arising from small edge-weights has a  negligible contribution. In addition, the network restricted to high edge-weights becomes 
sparser, 
and thus it is decomposed
into 
relatively small connected 
components. This idea has been used  in \cite{GN22,GHN22+} to study the largest eigenvalue of weighted  Erd\H{o}s-R\'{e}nyi graphs (i.e.~sparse Wigner matrices). However, since the underlying graph considered in this paper is a random $d$-regular graph which possesses a  strong dependency between entries, additional difficulties appear in the analysis. Also,  due to the special degree-structure of random $d$-regular graphs,  the governing spectral mechanism  is significantly different from the case  of  Erd\H{o}s-R\'{e}nyi graphs.

~

We now elaborate on this in more details. For a suitable sequence $\{b_n\}_{n\geq 1}$ slowly growing in $n$, we truncate edge-weights $W_{ij} $ as follows:
\begin{equation*}
W_{ij}^{(1)} \coloneqq W_{ij}\1_{|W_{ij}|>b_n^{1/\alpha}}, \quad
W_{ij}^{(2)} \coloneqq W_{ij}\1_{|W_{ij}|\le b_n^{1/\alpha}}.
\end{equation*}
This yields a decomposition of the underlying  random $d$-regular graph $A = A^{(1)} + A^{(2)}$ as
\begin{equation*}
A_{ij}^{(1)} = A_{ij}\1_{|W_{ij}|>b_n^{1/\alpha}} , \quad
A_{ij}^{(2)} = A_{ij}\1_{|W_{ij}|\le b_n^{1/\alpha}},
\end{equation*}
and a decomposition of the corresponding network $X=X^{(1)}+X^{(2)}$ as
\begin{equation*}
X^{(1)}_{ij} = A_{ij}^{(1)}W_{ij}, \quad X^{(2)}_{ij} = A_{ij}^{(2)}W_{ij}.
\end{equation*}
Since the largest eigenvalue of the random $d$-regular graph $A$ is $d$, if the truncation parameter $b_n$ satisfies $b_n \ll \log n$, then
  the network $X^{(2)} = A ^{(2)}\odot W   = A \odot W ^{(2)}$ is spectrally negligible. 

Next, note that given  $A^{(1)}$, edge-weights on the network $X^{(1)}$ are i.i.d.~Weibull distributions conditioned to be greater than $b_n^{1/\alpha}$ in absolute value.  In addition,
$A^{(1)}$  can be regarded as a (bond) percolation with a connectivity probability
\begin{align} \label{p}
p=\P(|W_{ij}|>b_n^{1/\alpha})\le  C_2e^{-b_n}
\end{align} 
 on the random $d$-regular graph $A$.  This percolation effect makes the graph $A$ much sparser, and thus the size of every connected component in $A^{(1)}$ becomes relatively small
 with high probability. Indeed, we establish a general fact that the percolation, whose connectivity probability satisfies \eqref{p} with $b_n$ suitably growing in $n$, on any  graph with \emph{uniformly bounded degrees} has a sharp shattering effect (see Lemma \ref{lem: size of connected compo} for details).

We aim to analyze the largest eigenvalue of each component in $A^{(1)}$. A crucial observation is that although the random $d$-regular graph $A$ can contain a number of cycles, every component in $A^{(1)}$ is ``almost'' tree. This is because 
the random $d$-regular graph is \emph{locally} a  $d$-regular tree and the size of every component in $A^{(1)}$  is relatively small, as explained in the previous paragraph. Inspired by this observation, we further decompose $X^{(1)} : = X^{(1,1)}+X^{(1,2)}$, where the former part consists of   vertex-disjoint trees and the latter part consists of tree-excess edges (i.e.~edges in a graph which do not belong to the spanning tree) in the network $X^{(1)}$. Since the number of tree-excess edges is relatively small,  $X^{(1,2)}$ is spectrally negligible.

Hence, it reduces to analyze the network $X^{(1,1)}$, consisting of  vertex-disjoint trees. By the above discussion, the size of each  tree is relatively small
and crucially, the maximum degree is bounded by $d$. We develop a general strategy to   bound the largest eigenvalue  of such tree-network  which we call $T$. The idea is to control the largest eigenvalue in terms of the ($\alpha$th-power of) $\ell^\alpha$-norm  of $T$, i.e.~(twice of) i.i.d.~sum of the $\alpha$th-power of  edge-weights $Y_{ij}$ in $T$, described as follows. By a  variational  formula for the largest eigenvalue,   denoting  by $\textbf{v} = (v_i)_{i\in V(T)}$  the top eigenvector of $T$ with $\norm{\textbf{v}}_2=1$, applying H\"{o}lder's inequality,
\begin{align} \label{144}
\lambda_{1}(T)  =  \sum_{(i,j)\in  \overrightarrow{E\;}(T)} Y_{ij}v_{i}v_{j} \le \Big( \sum_{(i,j)\in \overrightarrow{E\;}(T)} |Y_{ij}|^{\alpha} \Big)^{\frac{1}{\alpha}} \Big(  \sum_{(i,j)\in \overrightarrow{E\;}(T)} |v_{i}v_{j}|^{\beta} \Big)^{\frac{1}{\beta}},
\end{align}
where $\beta$ denotes the H\"{o}lder conjugate of $\alpha$ (i.e.~$\frac{1}{\alpha}+\frac{1}{\beta}=1$) and  $\overrightarrow{E\;}(T)$ denotes the collection of directed edges in $T$ (see \eqref{directed} for a precise definition). Since $T$ is a tree  whose maximum degree is at most $d$, $T$ can be embedded into the    infinite $d$-regular tree  $\mathbb{T}_d$. Hence,  \emph{uniformly} in $T$,  the second quantity  above  is bounded by
\begin{align} \label{145}
\sup_{\textup{\textbf{u}}=(u_i)_{i \in V(\mathbb{T}_d)}, \norm{\textup{\textbf{u}}}_2=1}    \Big(  \sum_{(i,j)\in \overrightarrow{E\;}(\mathbb{T}_d )} |u_{i}u_{j}|^{\beta} \Big)^{\frac{1}{\beta}}  \le C(\beta,d),
\end{align}
where $C(\beta,d)$ is the constant  related to the variational problem  \eqref{h}. Hence,   the inequalities \eqref{144} and \eqref{145} bound the largest eigenvalue in terms of the $\ell^\alpha$-norm of the tree-network.
Note that
a particular choice of the $\ell^\alpha$-norm is motivated by the fact that the i.i.d.~sum of the $\alpha$th-power of the reasonable number of Weibull random variables with a  shape parameter $\alpha$ has a similar tail behavior as that of  the $\alpha$th-power of a single  Weibull random variable.

However, a naive application of \eqref{144} does not work. This is because edge-weights on the graph $A^{(1,1)}$ are i.i.d.~Weibull distributions \emph{conditioned} to be greater than  $ b_n^{1/\alpha}$ in absolute value, which makes the tail of $\ell^\alpha$-norm of the network much heavier.
To remedy 
this problem, inspired by {\cite[Proposition 5.7]{GN22}}, we split the network according 
to the values of $v_i$s (recall that $\textbf{v} = (v_i)_i$ denotes the top eigenvector). We first show 
that the network induced 
by low values of $v_i$s is negligible. This is obtained by  deducing that the second quantity in RHS of \eqref{144} becomes relatively small on this network,   by crucially relying on the uniform boundedness of the maximum  degree.
Hence, it suffices to analyze the network 
arising from   vertices possessing high values of $v_i$s. A crucial aspect is that the size of this network is bounded only in terms of the  truncation value of $v_i$s, not depending on $b_n$. This makes
the aforementioned strategy \eqref{144} and \eqref{145} effectively work.

~

\noindent
\emph{Localization of top eigenvector} Let $\textup{\textbf{f}}=(f_i)_{i\in [n]}$ be a top eigenvector  with $\norm{\textbf{f}}_2=1$ and let us first consider the case $\alpha>2$. Since networks $X^{(2)}$ and  $X^{(1,2)}$ are spectrally negligible, the contribution arising from $X^{(1,1)}$, consisting of  vertex-disjoint trees, is ``almost'' at least $h_d(\alpha) (\log n)^{1/\alpha}.$ 
By the inequality \eqref{144}  applied to each tree-component in $X^{(1,1)}$, we deduce that for any $\e>0$, $\textbf{f}$ is  essentially supported on the components whose $\ell^\alpha$-norm of edge-weights is at least $ (1-\e)(\log n)^{1/\alpha}$. 
By a tail estimate on the $\ell^\alpha$-norm of  edge-weights, 
the number of such components is at most $n^{\alpha \e}$. Since the size of every component in $X^{(1,1)}$  is relatively small, we establish a localization of the top eigenvector $\textbf{f}$ on  $n^{\alpha \e + o(1)}$ vertices.

In the case of edge-weights possessing heavier tails (i.e.~$0<\alpha<2$), one can improve the localization result by analyzing the contribution of  networks in a refined fluctuation scale. To be precise, using the inequality $\lambda_1(X) \geq \max_{(i,j) \in E(A)}| W_{ij}|$,  we obtain that there is a constant $0<\tau<1$ such that with high probability, 
\begin{align*}
    \lambda_1(X) \geq  (\log n)^{1/\alpha}- \e(\log n)^{\tau/\alpha}.
\end{align*}
One can also  control the contribution of   negligible parts of $X$ (i.e.~$X^{(2)}$ and  $X^{(1,2)}$) in the scale $(\log n)^{\tau/\alpha},$ by exploiting the (quantitative) local tree-structure of random $d$-regular graphs. This allows us to deduce that
$\textbf{f}$ is essentially supported on the components whose $\ell^\alpha$-norm is at least $(\log n)^{1/\alpha} - \e^{1/2} (\log n)^{\tau/\alpha}$. 
Note that the number of such components is sub-polynomial. To see the structure of the top eigenvector, for each tree-component $T_k$ in  $X^{(1,1)}$ which mainly contributes to $\lambda_1(X)$, by a bound \eqref{144}, the quantity
\begin{align*}
    \Big(  \sum_{(i,j)\in \overrightarrow{E\;}(T_k)} |f_{i}f_{j}|^{\beta} \Big)^{\frac{1}{\beta}}
\end{align*}
is close to the maximum possible value (which is related to variation problem \eqref{h}). In the case $\alpha<2$ (equivalently, $\beta>2$), one can show that this is only possible when $|f_i|^2,|f_j|^2 \approx \norm{\textbf{f}}^2_2/2$ for some vertices $i$ and $j$ connected by an edge. This verifies the localization of the top eigenvector on sub-polynomial vertex-disjoint edges.

\subsection{Organization of paper}
The rest of the paper is structured as follows. In Section \ref{sec2}, we  prove a lower bound for 
the largest eigenvalue. In Section  \ref{sec3}, we present central and  ubiquitous strategies to upper bound the largest eigenvalue. Section \ref{sec4}  contains the proof of a localization  of the top eigenvector. Section \ref{sec5} includes some properties of the LLN constant of the largest eigenvalue. Appendix contains a technical estimate on tails of the sum of i.i.d.~Weibull random variables.

\subsection{Acknowledgement}
KN's research is supported by the National Research Foundation of Korea (NRF-2019R1A5A1028324, NRF-2019R1A6A1A10073887). JL's research is  supported by the Hong Kong Research Grants Council (GRF-16301519, GRF-16301520).  We  thank Charles Bordenave  for pointing out the reference \cite{bethe}. We  also thank Noga Alon and 
 Roland Bauerschmidt  for helpful comments.

\section{Lower bound on the largest eigenvalue}\label{sec2}

In this section, we lower bound the largest eigenvalue of the weighted random $d$-regular graph.

\subsection{Locally tree-like structure}
As mentioned in Section \ref{idea}, we lower bound the largest eigenvalue by constructing $\Theta(n)$ vertex-disjoint  $d$-regular trees  induced with high values of edge-weights. This can be accomplished using the fact that
the random $d$-regular graph is locally a $d$-regular tree.

Before elaborating on this, we introduce some notations. For $i,j\in V(G)$, denote by $\dist(i,j)$ the graph distance between $i$ and $j$, i.e.~the minimum length of a path joining $i$ and $j$ in a graph $G$.
For  $i\in V(G)$ and $R\in\mathbb{N}$, let $B_{R}(i)$, a $R$-neighborhood of  $i$, be a subgraph induced  by  the collection of vertices $\{j\in V(G):\dist(i,j)\le R\}$. 
Next, for a connected graph $G$, denoting by $T$ any spanning tree of $G$, edges in $E(G) \setminus E(T)$ are called \emph{tree-excess} edges. Note that the spanning tree $T$ may not be unique and thus tree-excess edges depend on the choice of the spanning tree. The \emph{excess} of $G$ is defined to be the number of  tree-excess edges, i.e.~$|E(G) \setminus E(T) | = |E(G)| - |V(G)| + 1,$ which does not depend on the choice of the spanning tree.

It is known  (see for example \cite[Proposition 4.1]{BHY19} and \cite{LS10}) that for a logarithmic radius $R_n := \lfloor c \log n \rfloor$ ($c>0$ is a constant), all $R_n$-neighborhoods of the random $d$-regular tree have a small number of tree-excess edges and most of them (except polynomially many) are trees. The following proposition is a version of  \cite[Proposition 4.1]{BHY19},  stated for general values of the radius $R_n$. This will be crucially used to  bound the largest eigenvalue in a sharp quantitative sense.
 
\begin{proposition}\label{prop: almost tree nbhd}
Let $G$ be a random $d$-regular graph with $|V(G)|=n$. Let $w\in \N$ be a constant and $\{R_n\}_{n\ge 1}$ be a sequence of positive integers such that $1\ll R_n \le 0.9\log_{d-1}n$.
 Define the events
 \begin{align}
\mathcal{E}_{1,1}:=\{\text{For all $i\in V(G)$,  the excess of $B_{R_n}(i)$ is at most $\omega$}\}
 \end{align}
and
 \begin{align}
  \mathcal{E}_{1,2}:=\{\text{$\size{\parenn{i\in V(G):\text{$B_{R_n}(i)$ contains a cycle}}}\le (d-1)^{4R_{n}}$}\}.
 \end{align}
 Then,   for sufficiently large $n$, 
 \begin{align}
     \P(\mathcal{E}_{1,1}^c) \le C n^{-w} (d-1)^{2R_n(w+1)}
 \end{align}
 and  there exists a constant $c>0$ such that
  \begin{align} \label{114}
     \P(\mathcal{E}_{1,2}^c) \le dn\exp(-c(d-1)^{2R_n}).
 \end{align}
Further assume that $\log\log n\ll R_n  \ll \log n.$  Then, defining the event   $\mathcal{E}_1 \coloneqq \mathcal{E}_{1,1} \cap \mathcal{E}_{1,2}$, for any  constant $\e>0$, for sufficiently large $n$,
 \begin{align} \label{final}
     \P(\mathcal{E}_{1}^{c})  \le n^{-w+\e}.
 \end{align}
\end{proposition}

Since Proposition \ref{prop: almost tree nbhd} can be immediately obtained by the argument of \cite[Proposition 4.1]{BHY19}, we provide a proof in the appendix.
As $B_R(i)$ is always connected, $B_R(i)$ not containing a cycle is equivalent to   $B_R(i)$ being a tree. Noting that every vertex has a degree $d$, this is further equivalent to the fact that $B_R(i)$ is a $d$-regular tree of depth $R$\footnote{A tree is said to be a $d$-regular tree of depth $L$ (with a distinguished vertex called as $\textsf{root}$) if all vertices except leaves have degree $d$ and every leaf $v$ satisfies $\dist(\textsf{root} ,v) = L$.} with a root $i$.

Proposition \ref{prop: almost tree nbhd} motivates to introduce a version of the quantity $h_d$ (defined in \eqref{h}) for a $d$-regular tree of ``finite'' depth, which will be done in the next section. 

\subsection{Finitary version of \eqref{h}}
We define a version of the variational problem  \eqref{h}  on the $d$-regular tree of \emph{finite} depth.
For any graph $G$, let
\begin{equation} \label{directed}
\overrightarrow{E\;}(G) \coloneqq \{ (i,j)\in V(G)\times V(G) : \text{vertices $i$ and $j$ are connected by an edge} \}
\end{equation}
be a collection of directed edges. Note that  $|\overrightarrow{E\;}(G)| = 2|E(G)|.$

For $L\in\mathbb{N}$, we denote by  $\mathbb{T}_{d}^{(L)}$  the $d$-regular tree of depth $L$. 
For $\gamma>0$,  define {
\begin{equation}\label{eq: C_alpha,L}
K^{(L)}_d(\gamma) \coloneqq
\sup_{\textbf{u}=(u_i)_{i\in V(\mathbb{T}_{d}^{(L)})}, u_i\geq 0, \lVert \textbf{u} \rVert_1 = 1}
\Big( \sum_{ (i,j)\in \overrightarrow{E\;}(\mathbb{T}_{d}^{(L)}) } u_{i}^{\gamma} u_{j}^{\gamma} \Big)^{\frac{1}{2\gamma}}.
\end{equation}} 
Since $\{K^{(L)}_d(\gamma) \}_{L\in\mathbb{N}}$ is   non-decreasing in $L$, one can define
\begin{equation}\label{eq: C_alpha}
K_d(\gamma) \coloneqq \lim_{L\to\infty} K_d^{(L)}(\gamma)
\end{equation}
(which a priori can be infinity).
Equivalently,
\begin{equation}  \label{k}
K_d(\gamma) := 
\sup_{\textbf{u}=(u_i)_{i\in V(\mathbb{T}_{d})}, u_i\geq 0,\lVert \textbf{u} \rVert_1 = 1}
\Big( \sum_{ (i,j)\in \overrightarrow{E\;}(\mathbb{T}_{d}) } u_{i}^{\gamma} u_{j}^{\gamma} \Big)^{\frac{1}{2\gamma}}.
\end{equation}
The quantities $h_d$ (defined in \eqref{h}) and  $K_d$ are related as follows: For  $\alpha>2$,  denoting  by $\beta$ the conjugate of $\alpha$, i.e.~$\frac{1}{\alpha}+ \frac{1}{\beta}=1$,  
\begin{align} \label{relation}
    h_d(\alpha) = 2^{1/\alpha} K_d (\beta/2).
\end{align}

Note that  for any $\gamma \ge \frac{1}{2}$, the quantity $K_d(\gamma)$ is finite. Indeed, since $u_{i}\in [0,1]$ in the variational problem \eqref{k}, 
    \begin{equation} \label{112}
        \sum_{ (i,j)\in \overrightarrow{E\;}(\mathbb{T}_{d}) } u_{i}^{\gamma} u_{j}^{\gamma} \le \sum_{ (i,j)\in \overrightarrow{E\;}(\mathbb{T}_{d}) } \sqrt{u_{i} u_{j}}.
    \end{equation}
    The RHS above is bounded by the spectral norm of the infinity $d$-regular tree $\mathbb{T}_{d}$, which is finite (in particular, $2\sqrt{d-1}$). See the discussion below \eqref{infinite} for details.

Although there is no exact  formula of the quantity $K_d(\gamma)$ for general values of  $\gamma,$  one can exactly solve  the variational problem \eqref{k} when $\gamma \ge 1.$
\begin{lemma} \label{formula k}
    Let  $d\ge 3$ be any integer. Then, for any $L\in \N$ and $\gamma \ge 1$,  $K_d ^{(L)}(\gamma)=2^{ \frac{1}{2\gamma}-1 }$, implying that $K_d (\gamma)=2^{ \frac{1}{2\gamma}-1 }$.
    If in particular $\gamma>1$, then  the maximizer in \eqref{eq: C_alpha,L} (or \eqref{k}) is only obtained when $u_i=u_j = 1/2$ (other $u_k$s are zero) for some vertices $i$ and $j $ connected by an edge. 
\end{lemma}
Note that in the case $\gamma=1$,  there are other types of the maximizer. For example, denoting by $0$ and $\{1,2,\cdots,d\}$ the root and its children respectively, any vector $\textbf{u}=(u_i)_{i}$ with $u_0=1/2$ and $u_1+\cdots+u_d=1/2$ (other $u_k$s are zero)  attains the maximum as well. In addition,  one can  indeed show that 
for any tree $T$ and $\gamma \ge 1$,
  \begin{align} \label{general}
        \sup_{ \textbf{u} = (u_i)_{i\in V(T)}, u_i\ge 0,   \norm{\textbf{u}}_1=1} \Big( \sum_{ (i,j)\in \overrightarrow{E}(T) } u_{i}^\gamma u_{j}^\gamma \Big)^{\frac{1}{2\gamma}} =  2^{ \frac{1}{2\gamma}-1 }
    \end{align}
 (see Remark \ref{remark53} for explanations).
 
~

The following lemma illustrates the structure of the ``near''-maximizer of the above variational problem  in a quantitative sense,  in the case $\gamma>1.$ This will play a crucial role when showing the  localization of the top eigenvector on the relatively small number of vertex-disjoint edges, in the case $0<\alpha<2$ (see Theorem \ref{thm: localization} for a precise statement).

\begin{lemma} \label{lemma 4.1}
Let $\gamma>1$. Then, there exists a constant $c = c(\gamma)>0$ such that the   following statement holds:  Let $T=(V,E)$ be any tree. For any small enough constant $\e>0$ (depending only on $\gamma$), if the vector  $\mathbf{u}=(u_{i})_{i \in V}$ with $\lVert \mathbf{u} \rVert_{1}=1$ and $u_i \ge 0$ satisfies
\begin{equation} \label{652}
\Big(\sum_{(i,j)\in \overrightarrow{E\;} } u_{i} ^{\gamma}u_{j} ^{\gamma}\Big)^{\frac{1}{2\gamma}}  \ge  (1-\e)\cdot 2^{ \frac{1}{2\gamma}-1 },
\end{equation}
 then there exist vertices $i_0$ and $j_0$ connected by an edge such that
 \begin{equation*}
  u_{i_0}, u_{j_0}  \ge  \frac{1}{2} -c {\sqrt{\ep}}.
\end{equation*}
\end{lemma}

Lemmas \ref{formula k} and \ref{lemma 4.1} will be proved in Section \ref{sec5} later.

\subsection{Lower bound on $\lambda_1(X)$}
Equipped with previous preparations, in this section, we  lower bound the largest eigenvalue.
In the following proposition, we provide a stretched exponential bound on the lower tail of the largest eigenvalue. Although Theorem \ref{thm: LLN} is stated for  $\alpha>2$   and $0<\alpha\le 2$ separately,  proofs are different depending on the cases $\alpha>1$ and $0<\alpha\le 1$, and thus we state the proposition accordingly.

\begin{proposition}[Lower tail upper bound]\label{prop: lower tail upper bound}
	Let $L\in\mathbb{N}$ and $\delta\in(0,1)$ be constants. Then, the following statements hold. \\
 1. In the case $\alpha > 1$, denoting by $\beta$ the H\"{o}lder conjugate of $\alpha$, for any small enough $\delta>0,$ 
	\begin{equation*}
	\P\paren{ \lambda_{1}(X) \le (1-\delta) 2^{1/\alpha}  K_d^{(L)} (\beta/2) (\log n)^{1/\alpha }}
\le  \exp(-n^{ 1 - (1-\delta)^{\alpha}+o(1)}).
	\end{equation*}
	2. In the case $0<\alpha\le 1$, 
	\begin{equation*}
	\P\paren{ \lambda_{1}(X) \le (1-\delta)(\log n)^{1/\alpha} }\le  \exp(-n^{ 1 - (1-\delta)^{\alpha}+o(1)}).
	\end{equation*}
\end{proposition}
 
\begin{proof}

$\empty$\\
\indent 
\textbf{Case 1.~$\boldsymbol{\alpha > 1}$.} For any constant $L\in\mathbb{N}$, we aim to take an enough number of  vertex-disjoint $d$-regular trees of depth $L$ in  the random $d$-regular graph $A$.
Let $w\coloneqq1$ and $R_n \coloneqq \lfloor \e \log_{d-1} n \rfloor$ ($\e>0$ is a small constant chosen later) in  Proposition \ref{prop: almost tree nbhd}, and consider the event $\cE_{1,2}$.
A vertex $i$ is called \emph{good} 
 if its $R_n$-neighborhood $B_{R_n}(i)$ is a tree, and let $S$ be a collection of good vertices.
 Under the event $\cE_{1,2}$,
 \begin{align} \label{s}
     |S| \geq n- (d-1)^{4R_n} \ge n- n^{4\e}.
 \end{align}  Also, by \eqref{114} in  Proposition \ref{prop: almost tree nbhd}, for  large $n$, 
 \begin{align} \label{621}
     \P(\cE_{1,2}^c) \le  dn\exp(-c(d-1)^{2R_n}) \le  \exp(-n^\e).
 \end{align}
 
 Choose any $i_{1}\in S$ and set $T_{1}\coloneqq B_{L}(i_{1})$. Since $i_1$ is a good vertex and $L<R_n$ for large $n$, $T_1$ is a $d$-regular tree of depth $L$ with a  root $i_{1}$.  Next, take any  $i_{2}\in S \setminus V(B_{2L}(i_{1}))$ and similarly set $T_{2}\coloneqq B_{L}(i_{2})$. Such choice of $i_2$ ensures that $T_1$ and $T_2$ are vertex-disjoint.

We proceed with  this procedure repeatedly: given $k$ chosen vertices $i_1,\cdots,i_k\in S$ such that $T_\ell\coloneqq B_L(i_\ell)$ are vertex-disjoint $d$-regular trees of depth $L$, take any $i_{k+1} \in  S \setminus  \cup_{\ell=1}^k V(B_{2L}(i_\ell) ) $.  Then, $T_{k+1} \coloneqq B_L(i_{k+1})$ is a 
 $d$-regular tree of depth $L$ vertex-disjoint from $T_1,\cdots,T_k.$  Since $|V(B_{2L}(i))| \leq 2d^{2L}$ for any vertex $i$ and recalling \eqref{s},  under the event $\cE_{1,2}$,
one can take  
\begin{equation} \label{ell}
m\coloneqq \left\lfloor \frac{n-n^{4\e}}{2d^{2L}} \right\rfloor
\end{equation}
vertex-disjoint $d$-regular trees  $T_1,T_2,\cdots,T_m $ of depth $L$. Regarding $T_k$ as a tree-network induced with edge-weights $W_{ij}$, vertex-disjointness implies that
\begin{align*}
    \lambda_{1}(X)\ge\max_{1\le k\le m}\lambda_{1}(T_{k}) .
\end{align*}
Since   $\lambda_1(T_k) $s are  conditionally  independent given  $A$,
\begin{align} \label{eq: lower tail ineq 1}
    \P\big( \lambda_{1}(X) & \leq (1-\delta)  2^{1/\alpha} K_d^{(L)} (\beta/2) 
 (\log n)^{1/\alpha } \big) \nonumber \\
    &\leq  \E \big[  \P \big(  \lambda_{1}(X) \leq (1-\delta)  2^{1/\alpha} K_d^{(L)} (\beta/2) 
 (\log n)^{1/\alpha } \mid A \big) \1_{\cE_{1,2}} \big] + \P(\cE_{1,2}^c)    \nonumber \\
&\leq  \E \Big[\prod_{k=1}^m  \P \big(  \lambda_{1}(T_{k})\leq  (1-\delta) 2^{1/\alpha} K_d^{(L)} (\beta/2) 
 (\log n)^{1/\alpha } \mid A \big) \1_{\cE_{1,2}} \Big] +\P(\cE_{1,2}^c)    .
\end{align}

We now  bound $\lambda_1(T_k)$ for each $k$. 
Let $\textbf{u}=(u_i)_{i\in V(T_k)}$ be a vector realizing the supremum in \eqref{eq: C_alpha,L} with $\gamma := \beta/2$ (recall that $T_k$ is isomorphic to $\mathbb{T}_{d}^{(L)}$, a $d$-regular tree of depth $L$).   Then,  the vector $\textbf{v}=(v_i)_{i\in V(T_k)}$ defined by $v_i \coloneqq \sqrt{u_i}$ satisfies $ \norm{\textbf{v}}_2=1$ and
\begin{align} \label{kd}
  K_d^{(L)} (\beta/2) = 
\Big( \sum_{ (i,j)\in \overrightarrow{E\;}(T_k) } v_{i}^{\beta} v_{j}^{\beta} \Big)^{\frac{1}{\beta} }.
\end{align}
Let
$\tau\coloneqq2 (1-\delta)^{\alpha}K_d^{(L)} (\beta/2)^{-\beta }   $.
Conditionally on the event $\cE_{1,2}$, for each $k=1,\cdots,m$, define the event  
\begin{equation*}
\cF_{k}\coloneqq \left\{W_{ij} \ge 
 (\tau    v_{i}^\beta  v_{j}^{\beta}  \log n )^{1/\alpha }  \text{ for all }(i,j)\in E(T_{k}) \right\}.
\end{equation*}
Recalling $ \norm{\textbf{v} }_2=1$, by the variational formula  for the largest eigenvalue (see Lemma \ref{variation} in Appendix), under the event $\cF_{k}$,
\begin{align*}
\lambda_{1}(T_{k})\ge \sum_{(i,j)\in \overrightarrow{E\;}(T_{k})}W_{ij}v_{i}v_{j} &\ge  \tau^{1/\alpha }  \Big( \sum_{(i,j)\in \overrightarrow{E\;}(T_{k})} (v_{i}v_{j})^{1+\frac{\beta}{\alpha}}  \Big) (\log n)^{1/\alpha }  \\
&\overset{\eqref{kd}}{=}\tau^{1/\alpha } K_d^{(L)} (\beta/2)^{\beta } (\log n)^{1/\alpha } = (1-\delta)2^{1/\alpha} K_d^{(L)} (\beta/2) 
 (\log n)^{1/\alpha } ,
\end{align*}
where we used the fact $1+\frac{\beta}{\alpha} = \beta$. In addition,  by  the independence of $W$ and $A$,
\begin{align*}
    \P(\cF_{k}  \mid  A)& = \prod_{(i,j)\in E(T_{k})} \P\paren{ W_{ij}  \ge  (\tau v_{i}^\beta  v_{j}^{\beta}  \log n)^{1/\alpha} {\,|\, A}}\\
&\ge C\exp \Big ( -\tau  \log n  \sum_{(i,j) \in E (T_{k})} v_{i}^\beta  v_{j}^{\beta}   \Big ) \overset{\eqref{kd}}{=}  C\exp(-\tau  K_d^{(L)} (\beta/2)^{\beta }  \log n/2) = n^{- (1-\delta)^\alpha + o(1) } 
\end{align*}
($C = C(C_1,d,L)>0$ is a constant),
where the summation above is taken over undirected edges, which yields a division by 2 in the exponent. 
Hence,  under the event $\cE_{1,2}$,
\begin{equation*}
\P\paren{ \lambda_{1}(T_{k}) \ge (1-\delta) 2^{1/\alpha} K_d^{(L)} (\beta/2) 
 (\log n)^{1/\alpha } \mid A } \ge \P(\mathcal{F}_k \mid A) \ge   n^{- (1-\delta)^\alpha + o(1)}.
\end{equation*}
Applying this  to \eqref{eq: lower tail ineq 1}, recalling $m\asymp n$ for small enough $\e>0$ (say, $\e=1/8$, see \eqref{ell}),
 \begin{align} \label{eq: lower tail finish 1}
    \P\big( \lambda_{1}(X)  \leq (1-\delta)  2^{1/\alpha} K_d^{(L)} (\beta/2) 
 (\log n)^{1/\alpha } \big) \le \exp(-n^{1-(1-\delta)^{\alpha} + o(1)}) + \P(\cE_{1,2}^c)    .
\end{align}
Combining this with \eqref{621}, we  conclude the proof.

~

\textbf{Case 2.~$\boldsymbol{0 < \alpha \le 1}$.} 
Since $X$ is a symmetric matrix {with zero diagonal,} 
\begin{align} \label{large}
    \lambda_{1}(X)\ge \max_{1\leq i<j\leq n} |X_{ij}| =  \max_{(i,j)\in E(A)}|W_{ij}|
\end{align} (see  Lemma \ref{variation2} for details). Thus,  by  the independence of $W$ and $A$, along with  the fact  $|E(A)|=\frac{nd}{2}$,
\begin{align}\label{eq: lower tail finish 2}
    \P\big(\lambda_{1}(X) \leq  (1-\delta) (\log n)^{1/\alpha }  \big) &\le \P\paren{ \max_{(i,j)\in E(A)}|W_{ij}| \leq (1-\delta) (\log n)^{1/\alpha } } \nonumber \\
&\le \big( 1 - C_{1} n^{-(1-\delta)^{\alpha}} \big)^{\frac{nd}{2}} = \exp(-n^{ 1 - (1-\delta)^{\alpha}+o(1)}).
\end{align}

\end{proof}

Since $L \in \N$ in Proposition \ref{prop: lower tail upper bound} is arbitrary and $K_d^{(L)} (\beta/2) \uparrow K_d(\beta/2)$ as $L\rightarrow \infty$, when $\alpha>1$, for any small constant $\delta>0,$
\begin{align*}
\P( \lambda _1(X) \le (1-\delta)   2^{1/\alpha} K_d(\beta/2) (\log n)^{1/\alpha } ) \le  {\exp(-n^{ 1 - (1-\delta/2)^{\alpha}+o(1)})}.
\end{align*}
By the relation \eqref{relation}, we deduce that for $\alpha>2$,
\begin{align} \label{640}
\P( \lambda _1(X) \le (1-\delta)  h_d(\alpha) (\log n)^{1/\alpha } ) \le  {\exp(-n^{ 1 - (1-\delta/2)^{\alpha}+o(1)})}.
\end{align} 
In addition, when $0<\alpha\le 2,$  as a consequence of \eqref{large} (and by the argument in \eqref{eq: lower tail finish 2}), 
\begin{equation}   \label{641}
    \P\big(\lambda_{1}(X) \leq  (1-\delta) (\log n)^{1/\alpha }\big)  \le {\exp(-n^{ 1 - (1-\delta)^{\alpha}+o(1)})}.
\end{equation}
Indeed, as explained in the following remark, one can lower bound $\lambda_1(X)$ in a more refined scale. This plays a crucial role when improving the localization result of the top eigenvector in Section \ref{sec4}.

\begin{remark}[Lower tail fluctuation scale] \label{finer lower}
When $0<\alpha\le 2$, for any constant $\tau \in ({\max} \{0,1-\alpha \} ,1)$,
\begin{equation} \label{620}
    \P\big(\lambda_{1}(X) \leq  (\log n)^{1/\alpha } - (\log n)^{\tau/\alpha}  \big)  \le {\exp(-n^{ {\min(\alpha,1)\cdot} (\log n)^{(\tau-1)/\alpha} + O(1/\log n) })} = o(1).
\end{equation}
 To see this, using \eqref{large}, setting $ \delta_{n} \coloneqq (\log n)^{(\tau-1)/\alpha}$, the above probability is bounded by
\begin{align*}
   \P\paren{ \max_{(i,j)\in E(A)}|W_{ij}| \leq (\log n)^{1/\alpha } - (\log n)^{\tau/\alpha} } 
    &\le \exp(-n^{ 1 - (1-\delta_{n})^{\alpha} + O(1/\log n) }) \\
    &\le \exp(-n^{ {\min(\alpha,1)\cdot} \delta_{n} + O(1/\log n) }),
\end{align*}
where in the last inequality we used the fact $(1-x)^\alpha \le 1-{\min(\alpha,1)\cdot}x$ for $0\le x\le 1$.
Since  $\delta_{n}\gg (\log n)^{-1}$ (recall $\tau>{\max}\{0,1-\alpha \}$), the above bound is $o(1).$

In addition, when $\alpha>2,$ we similarly deduce from \eqref{eq: lower tail finish 1} that  (note that the estimate  \eqref{eq: lower tail finish 1} holds for varying $\delta$ as well)  for any constant  $\tau>0,$
\begin{align*}
     \P\big( \lambda_{1}(X)  \leq 2^{1/\alpha} K_d^{(L)}(\beta/2)
 (\log n)^{1/\alpha} - (\log n)^{\tau/\alpha} \big)  = o(1).
\end{align*}
However, unless $\tau \ge 1,$ this is \emph{not} enough to obtain the  following fluctuation result
 \begin{equation*}
    \P\big( \lambda_1(X) \le h_d(\alpha)
 (\log n)^{1/\alpha} - (\log n)^{\tau/\alpha} \big)  =o(1),
 \end{equation*}
 since we are not aware of the quantitative convergence rate $K^{(L)}_d \uparrow K_d$ as $L\to \infty$. The best (lower) fluctuation upper bound of $\lambda_1(X)$, which can be obtained by our approach, is of order $(\log n)^{1/\alpha}$, i.e.~the estimate \eqref{640}.
 
\end{remark}

\section{Upper bound on the largest eigenvalue} \label{sec3}

In this section,  we establish an upper bound for the upper tail of the largest eigenvalue.

\begin{proposition}[Upper tail upper bound]\label{prop: upper tail upper bound}
The following upper tail estimates hold.

    \noindent
	1. In the case $\alpha > 1$, for any  $\tau> \frac{\alpha+1}{2\alpha+1} $, denoting by $\beta$ the conjugate of $\alpha,$
	\begin{equation} \label{301}
	\lim_{n\to \infty} \P\paren{ \lambda_{1}(X) \ge 2^{1/\alpha} K_d (\beta/2) (\log n)^{1/\alpha } + (\log n)^{\tau/\alpha} } = 0.
	\end{equation}
 	2. In the case $0<\alpha\le 1$, for any  $\tau> \frac{2}{\alpha+2}  $, 
	\begin{equation} \label{302}
	\lim_{n\to \infty} 
 \P\paren{ \lambda_{1}(X) \ge  (\log n)^{1/\alpha } + (\log n)^{\tau/\alpha} }=0.
	\end{equation}
\end{proposition}

\begin{remark}[Upper tail fluctuation scale]
Although it suffices to verify
\begin{align*}
\begin{cases}
    \alpha >2:\lim_{n\to \infty} \P\paren{ \lambda_{1}(X) \ge h_{d}(\alpha) (\log n)^{1/\alpha } + \e (\log n)^{1/\alpha} } = 0, \\
 \alpha<2  : \lim_{n\to \infty} \P\paren{ \lambda_{1}(X) \ge   (\log n)^{1/\alpha } + \e (\log n)^{1/\alpha} } = 0
\end{cases}
\end{align*}
($\e>0$ is constant)  to prove the upper bound in Theorem \ref{thm: LLN},
we establish an upper bound in a finer fluctuation scale, since this will play a crucial role in establishing a sharp localization result for the top eigenvector.

In the case of unweighted random $d$-regular graphs (i.e.~``$\alpha=\infty$''), it is a long-standing conjecture that the (non-trivial) extreme eigenvalue $\lambda_2$ satisfies $n^{2/3} [{(d-1)^{-1/2}}\lambda_2 -2 ] - c_{n,d} \overset{\text{d}}{\to} \text{Tracy-Widom}$ ($ c_{n,d}>0$ is some constant). This in particular would imply that the typical value and fluctuations 
of the extreme eigenvalue are of order 1 and $n^{-2/3}$ respectively, provided that $c_{n,d} = O(1)$. Recently, as a crucial step towards this conjecture, Huang-Yau \cite{HY21+} proved that the fluctuation is at most 
polynomial order, i.e.~there exists a small enough constant $\tau>0$ such that $\lambda_2 \le 2{\sqrt{d-1}}+n^{-\tau}$ with high probability. Very recently, the conjectured Tracy-Widom fluctuation behavior was rigorously verified \cite{he2022spectral, HY23+} in certain ``dense'' regimes, i.e.~$d\to \infty $ as $n\to \infty$  (precisely, $n^{\ep}\le d\le n^{\frac{1}{3}-\ep}$ or $n^{\frac{2}{3}+\ep}\le d\le n$ for some $\e>0$).

Once  the random $d$-regular graph  is induced with Weibull weights with  a shape parameter $\alpha>1$, our results (Theorem  \ref{thm: LLN} and Proposition \ref{prop: upper tail upper bound}) rigorously justify that  the typical value and 
 (upper) fluctuations 
of the extreme eigenvalue are of order  $(\log n)^{  \frac{1}{\alpha} }$   and at most $(\log n)^{   \frac{1}{\alpha}\cdot \frac{\alpha+1}{2\alpha+1} }$  respectively. 
 It is an important and interesting problem to find  the correct order of  fluctuations and characterize the limiting distribution after normalization.
 
\end{remark}

In order to prove Proposition \ref{prop: upper tail upper bound}, as mentioned in Section \ref{idea}, we proceed in the following  steps:
\begin{enumerate}[(1)]
\item \label{outline_light_step_decompose}
For a suitable sequence $\{b_{n}\}_{n \ge 1},$ we truncate edge-weights $W_{ij} $ as follows:
\begin{equation*}
W_{ij}^{(1)} \coloneqq W_{ij}\1_{|W_{ij}|>b_{n}^{1/\alpha}}, \quad
W_{ij}^{(2)} \coloneqq W_{ij}\1_{|W_{ij}|\le b_{n}^{1/\alpha}}.
\end{equation*}
This yields a decomposition of the underlying  random $d$-regular graph $A = A^{(1)} + A^{(2)}$ as
\begin{equation*}
A_{ij}^{(1)} = A_{ij}\1_{|W_{ij}|>b_n^{1/\alpha}} , \quad
A_{ij}^{(2)} = A_{ij}\1_{|W_{ij}|\le b_n^{1/\alpha}},
\end{equation*}
and a decomposition of the corresponding network $X=X^{(1)}+X^{(2)}$ as
\begin{equation*}
X^{(1)}_{ij} = A_{ij}^{(1)}W_{ij}, \quad X^{(2)}_{ij} = A_{ij}^{(2)}W_{ij}.
\end{equation*}
Note that given  $A^{(1)}$, the edge-weights on the network $X^{(1)}$ are i.i.d.~Weibull distributions conditioned to be greater than $b_{n}^{1/\alpha}$ in absolute value.

\item \label{outline_light_step_component_analysis} 
 We analyze the component structure of $A^{(1)}$. By a tail decay of Weibull distributions, $A^{(1)}$  can be regarded as a (bond) percolation with a connectivity probability
\begin{align} \label{qq}
p=\P(|W_{ij}|>b_{n}^{1/\alpha})\le C_2e^{-b_n}
\end{align} 
 on the random $d$-regular graph $A$. We show that this percolation has a sparsification effect, which makes the size of every connected component in $A^{(1)}$  relatively small with high probability (see Lemma \ref{lem: size of connected compo} for details).
 
\item We further decompose  the random graph $A^{(1)}=A^{(1,1)}+A^{(1,2)}$, where $A^{(1,1)}$ consists of  several vertex-disjoint trees and $A^{(1,2)}$ consists of the tree-excess edges ($X^{(1,1)}$ and $X^{(1,2)}$ denote the corresponding networks). As mentioned in Section \ref{idea}, this decomposition allows one to effectively control the contribution of $X^{(1,1)}$ in terms of a universal object, the infinite $d$-regular tree. The network 
 $X^{(1,2)}$ is shown to be negligible, since the number of tree-excess edges is relatively small (see Proposition \ref{prop: almost tree nbhd}).

\item In order to control the main part $X^{(1,1)}$, consisting of vertex-disjoint trees,  we devise a general strategy to bound the largest eigenvalue of such tree-networks (see Proposition \ref{keyprop}). A crucial feature of the network of consideration is that the maximum degree is uniformly bounded, since $A^{(1,1)}$ is a subgraph of the random $d$-regular graph.  This plays a crucial role in bounding the largest eigenvalue in terms of the $\ell^\alpha$-norm of the conductance matrix in a sharp way.

\end{enumerate}

\subsection{Percolation on graphs with uniformly bounded degrees}
Recall that $A^{(1)}$  can be regarded as a (bond) percolation with a connectivity probability
$p\le C_2e^{-b_n}$
 on the random $d$-regular graph $A$ (see \eqref{qq}).  
 In the following lemma, we establish a connectivity property of  the percolation on the general graph having uniformly bounded degrees.

\begin{lemma}\label{lem: size of connected compo}
Let $C>0$, $d\in \N$ be constants and $\{b_n\}_{n\geq 1}$ be a sequence such that $b_n \gg \log \log n$.
Let $G$ be any graph with $|V(G)|=n$ whose maximum degree is at most $d$.
Consider the (bond) percolation on $G$ with connectivity probability $p \leq C e^{-b_n},$ which yields a (random) subgraph  $G'$ of $G$ consisting of open edges. Then, for sufficiently large $n,$  with probability at least $1-n^{-1}$, the  number of edges in every connected component in $G'$ is at most $ 
 \lfloor  3\log n/b_{n} \rfloor $.

	\begin{proof}
		For any $v\in V(G)$, the event that the connected component in $G'$, including $v$, has {at least $\ell$ edges} implies the following: There exist open edges $e_1,\cdots,e_\ell$ such that $v$ is one of  the endpoints of $e_1$ and, for $2\le i\le \ell$, $e_i$ is connected to  the subgraph induced by  edges $e_1,\cdots, e_{i-1}$. Since $|V(e_1\cup \cdots \cup e_{i-1})| \le i$ and degrees are bounded by $d$, by a union bound over all such possible collection of edges,
  \begin{align*}
      \P&(\exists 
 \ \text{connected component including $v$ has at least $\ell$ edges}) \\
  &\le ( d \cdot 2d \cdot 3d \cdot \cdots \cdot \ell d ) \cdot p^{\ell} \le d^{\ell}\ell^{\ell}C^{\ell}e^{-b_n \ell}.
  \end{align*}
   By a union bound over all vertices $v$ in $V(G)$,
   \begin{equation*}
       \P(\exists 
  \ \text{connected component having at least $\ell$ edges}) \le n \cdot d^{\ell}\ell^{\ell}C^{\ell}e^{-b_n \ell}.
   \end{equation*}
Taking  $\ell\coloneqq\lfloor 3\log n/b_{n} \rfloor +1$, using the condition $b_n \gg \log \log n$, we deduce that the above bound is at most $n^{-1}.$
   
	\end{proof}
\end{lemma}

\subsection{Largest eigenvalue of tree-networks}
 Recall that from the aforementioned idea of proofs, it is crucial to analyze the spectral behavior of  $X^{(1,1)}$, a collection of 
 vertex-disjoint trees. By Lemma \ref{lem: size of connected compo}, the size of every connected component in  $X^{(1)}$  (and thus in $X^{(1,1)}$) is relatively small with high probability.
In the following key proposition, we provide a general tail bound on the largest eigenvalue of such tree-networks whose degrees are uniformly bounded. 
\begin{proposition} \label{keyprop}
Let $c>0$, $d \in \mathbb{N}$ be constants and  $\{b_n\}_{n\geq 1}$ be  a sequence such that $\log \log n \ll b_n = O(\log n) $.
Let  $G=(V,E,Y)$ be a  tree-network $(Y = (Y_{ij})_{i,j\in V}$ denotes the conductance matrix$)$ satisfying {the following properties:}
\begin{enumerate}
\item[(i)] Maximum degree is at most $d$.
\item[(ii)]  $|V| \leq c\frac{\log n }{b_n}$.  
\end{enumerate}
Suppose that edge-weights $Y_{ij}$ are i.i.d.~Weibull distributions  with a shape parameter $\alpha>0$ conditioned to be  greater than  $b_{n}^{1/\alpha}$
in absolute value. Then, the following statements hold: 
\begin{enumerate}[1.]
\item In the case $\alpha>1$, for any sequence $\{a_n\}_{n\geq 1}$  satisfying $a_n  \gg  b_n^{\alpha / (\alpha+1)} (\log n)^{1/(\alpha+1)}$,  denoting by $\beta>1$ the H\"{o}lder conjugate of $\alpha,$   
\begin{align}\label{keydisp}
\mathbb{P} \left ( \lambda_1 \geq 2^{1/\alpha} K_d (\beta/2) (\log n)^{1/\alpha } +a_n^{1/\alpha} \right ) = o(n^{-1}),
\end{align}
where $K_d$ is the function defined in  \eqref{k}.

\item  In the case $0<\alpha\leq 1$, for any sequence $\{a_n\}_{n\geq 1}$  satisfying   $a_n \gg  b_n^{\alpha/2} (\log n)^{1-(\alpha/2)}$,  
\begin{align} \label{keydisp2}
\mathbb{P} \left ( \lambda_1 \geq (\log n)^{1/\alpha} + a_n^{1/\alpha} \right ) = o(n^{-1}).
\end{align}
\end{enumerate}

\end{proposition}

\begin{proof}
  Denote by $\textbf{f}=(f_{i})_{i\in V}$ a unit eigenvector corresponding to the largest eigenvalue, i.e.~$\norm{\textbf{f}}_2=1$ and
\begin{align*}
\lambda_1 =  \sum_{(i,j)\in \overrightarrow{E}} Y_{ij}f_{i}f_{j},
\end{align*}
where $\overrightarrow{E}$ denotes the set of directional version of edges in $G$ (i.e.~every edge is counted twice).

We decompose the network according to the value of $f_i$s with respect to the 
 truncation level $\eta_{n}>0$ satisfying  
\begin{equation} \label{312}
a_{n}\eta_{n}^{-\min\{2,2\alpha\}} \gg\log n,\quad  (\log n)^{1-1/\alpha}a_{n}^{1/\alpha}\gg  \eta_{n}^{-2}b_n.
\end{equation}
Such $\eta_n>0$ exists due to the  condition imposed on the sequences  $\{a_n\}_{n\geq 1}$ and  $\{b_n\}_{n\geq 1}$. 
Since $b_n = O(\log n)$ and the probabilities in \eqref{keydisp} and \eqref{keydisp2} are monotone in $a_n$, without loss of generality, we may assume that $a_n  = O(\log n)$. Since $b_n \gg \log \log n,$  by the second condition  in \eqref{312},
\begin{align} \label{313}
    \eta_n^{-2} \ll  \frac{ (\log n)^{1-1/\alpha}a_{n}^{1/\alpha}}{\log \log n}  = O\Big(\frac{\log n}{\log \log n}\Big).
\end{align}
 
Define the collection of vertices whose corresponding value of $f_i$  is less than $\eta_n$ in absolute value:
\begin{equation}\label{eq: index set for small coordinates}
I \coloneqq \{ i\in V : |f_{i}|<\eta_{n}\}.
\end{equation}
Then, since $\norm{\textbf{f}}_2=1$,
\begin{align} \label{321}
    |I^c| \le \eta_{n}^{-2}.
\end{align}
Now, we decompose the collection of directed edges $\overrightarrow{E\;}$ as
\begin{equation}\label{eq: E_S E_L}
\overrightarrow{E}_{S}\coloneqq \{ (i,j)\in\overrightarrow{E\;}  : i,j\in I \},\quad \overrightarrow{E}_{L}\coloneqq \overrightarrow{E\;}\backslash\overrightarrow{E}_{S},
\end{equation}
and also denote by $E_{S}$ and $E_{L}$ the undirected counterpart of $\overrightarrow{E}_{S}$ and $\overrightarrow{E}_{L}$ respectively. We write $\lambda_{1}= \lambda_{S} + \lambda_{L}$ as
\begin{equation}\label{eq: lambda_S lambda_L}
\lambda_{S} \coloneqq \sum_{(i,j)\in \overrightarrow{E}_{S}} Y_{ij}f_{i}f_{j} = 2\sum_{(i,j)\in E_{S}} Y_{ij}f_{i}f_{j}, \quad
\lambda_{L} \coloneqq \sum_{(i,j)\in \overrightarrow{E}_{L}} Y_{ij}f_{i}f_{j}
= 2\sum_{(i,j)\in E_{L}} Y_{ij}f_{i}f_{j}.
\end{equation}

\textbf{Case 1.~$\mathbf{\boldsymbol{\alpha}>1}$.}
Since $\lambda_{1}=\lambda_{S}+\lambda_{L}$, 
\begin{multline}\label{eq: upper tail light 3}
\P\paren{ \lambda_{1} > 2^{1/\alpha} K_d (\beta/2) (\log n)^{1/\alpha}+ a_{n}^{1/\alpha} } \\
\le \P\paren{ \lambda_{S} > 2^{-1}a_{n}^{1/\alpha} }
+ \P\paren{ \lambda_{L} > 2^{1/\alpha} K_d (\beta/2) (\log n)^{1/\alpha} + 2^{-1}a_{n}^{1/\alpha} }.
\end{multline}
We first  bound $\lambda_S$. By H\"{o}lder's inequality,
\begin{equation*}
\lambda_{S}  \le \Big( 2 \sum_{(i,j)\in E_{S}} |Y_{ij}|^{\alpha} \Big)^{\frac{1}{\alpha}} \Big(2 \sum_{(i,j)\in E_{S}} |f_{i}f_{j}|^{\beta} \Big)^{\frac{1}{\beta}}.
\end{equation*}
The second term above can be bounded via
\begin{align} \label{334}
2\sum_{(i,j)\in E_{S}} |f_{i}f_{j}|^{\beta} = 2\sum_{(i,j)\in E_{S}} |f_{i}f_{j}|^{\beta-1} |f_{i}f_{j}| 
&\le \  \sum_{(i,j)\in E_{S}} 
 |f_{i}f_{j}|^{\beta-1} (|f_{i}|^{2}+|f_{j}|^{2}) \nonumber  \\
&\le  \eta_{n}^{2(\beta-1)}  \sum_{(i,j)\in E_{S}} (|f_{i}|^{2}+|f_{j}|^{2}) \leq d\eta_{n}^{2(\beta-1)}  ,
\end{align}
where we used the condition (i) together with  $\norm{\textbf{f} }_2=1$ in the last inequality.  
Hence, by Lemma \ref{lem: conditioned alpha power sum asymp} together with the fact $|E|= |V|-1 \le c\log n/b_{n}$,
\begin{multline}\label{eq: upper tail light 4}
\P (\lambda_{S} > 2^{-1}a_{n}^{1/\alpha} ) \le \P\Big( \sum_{(i,j)\in E_S} |Y_{ij}|^{\alpha} >  \frac{2^{-\alpha-1}a_{n}}{d^{\alpha/\beta}\eta_{n}^{2}} \Big) 
\le \P\Big( \sum_{(i,j)\in E} |Y_{ij}|^{\alpha} >  \frac{2^{-\alpha-1}a_{n}}{d^{\alpha/\beta}\eta_{n}^{2}} \Big) \\
\le \exp\paren{-\frac{2^{-\alpha-1}a_{n}}{d^{\alpha/\beta}\eta_{n}^{2}} + c\log n + {o(\log n)}}=o(n^{-1})
\end{multline}
(Lemma \ref{lem: conditioned alpha power sum asymp} is applicable since $ a_{n}\eta_{n}^{-2} \gg \log n \gg  \log n / b_n$, see \eqref{312}),
where  in the last estimate we used again the first condition in  \eqref{312}.

Next, we  bound $\lambda_{L}$. By  H\"{o}lder's inequality,
\begin{align*}
\lambda_{L} \le \Big( 2\sum_{(i,j)\in E_{L}} |Y_{ij}|^{\alpha} \Big)^{\frac{1}{\alpha}} \Big(  2\sum_{(i,j)\in E_{L}} |f_{i}f_{j}|^{\beta} \Big)^{\frac{1}{\beta}}.
\end{align*}
Since $G$ is a tree with maximum degree at most $d$, it can be embedded into the infinite $d$-regular tree $\mathbb{T}_d$. Thus, by the definition of the function $K_d$  in  \eqref{k},
\begin{align}  \label{330}
   \Big( 2\sum_{(i,j)\in E_{L}} |f_{i}f_{j}|^{\beta} \Big)^{\frac{1}{\beta}} \le   \Big(2 \sum_{(i,j)\in E} |f_{i}f_{j}|^{\beta} \Big)^{\frac{1}{\beta}} \le   
 \sup_{ \norm{\textbf{v}}_2=1} \Big( 2\sum_{(i,j)\in E(\mathbb{T}_d)} |v_{i}v_{j}|^{\beta} \Big)^{\frac{1}{\beta}} = K_d(\beta/2),
\end{align} 
 where the last equality follows from the observation that the constraint above is w.r.t. the  $\ell^2$-norm $\norm{\textbf{v}}_2=1$.
Hence, by the above displays,
\begin{align} \label{355}
     \lambda_L\le 2^{1/\alpha} K_d(\beta/2)  \Big( \sum_{(i,j)\in E_{L}} |Y_{ij}|^{\alpha} \Big)^{\frac{1}{\alpha}} .
\end{align}
Note that the event $
\big\{ \sum_{(i,j)\in E_L} |Y_{ij}|^{\alpha} > t  \big\}$
implies the existence of a ({random}) subset $J\subseteq V$ with $|J|\le\eta_{n}^{-2}$ such that  $
\sum_{(i,j)\in E,  \ i  \ \text{or} \  j \in J} |Y_{ij}|^{\alpha} > t$. This yields the entropy factor bounded by $(c \log n/b_{n}) ^{\eta_{n}^{-2}}$.
Since
\begin{align} \label{323}
    |E_{L}|\le d |I^c| \overset{\eqref{321}}{\le}  d\eta_{n}^{-2}
\end{align}
(the condition (i) is used in the first inequality), by \eqref{355} and  Lemma \ref{lem: conditioned alpha power sum asymp}  
together with a union bound, for  some constant $c'>0$,
\begin{align}\label{eq: upper tail light 5}
\P&\paren{ \lambda_{L} > 2^{1/\alpha} K_d (\beta/2)(\log n)^{1/\alpha} + 2^{-1}a_{n}^{1/\alpha} } \nonumber \\
&\le ( c\log n/b_{n}) ^{\eta_{n}^{-2}}
\exp\big(- \log n - c' (\log n)^{1-1/\alpha}a_{n}^{1/\alpha} + d\eta_{n}^{-2}b_{n} + o(\eta_{n}^{-2}b_{n})\big) = o(n^{-1})
\end{align}
(Lemma \ref{lem: conditioned alpha power sum asymp} is applicable since $\log n\gg \eta_{n}^{-2}$, see \eqref{313}).
Here, we used {$(x+y)^\alpha \ge x^\alpha + \alpha x^{\alpha-1} y$ 
 ($x,y\ge 0$ and $\alpha>1$)} in the first inequality. Also, in the last estimate above, we used the second condition in \eqref{312} and the fact
\begin{align} \label{entropy}
    ( c\log n/b_{n}) ^{\eta_{n}^{-2}} = \exp ( o((\log n)^{1-1/\alpha}a_{n}^{1/\alpha}))
\end{align}  
which follows from the first inequality in \eqref{313}.

Therefore, \eqref{eq: upper tail light 4} and \eqref{eq: upper tail light 5} conclude the proof.

 ~

\textbf{Case 2.~$\boldsymbol{0 < \alpha \le 1}$.}
Similarly as above,
\begin{align}\label{eq: small + large}
\P\paren{ \lambda_{1}  >  (\log n)^{1/\alpha} + a_{n}^{1/\alpha} }
\le \P\paren{ \lambda_{S} > 2^{-1}a_{n}^{1/\alpha} }
+ \P\paren{ \lambda_{L} > (\log n)^{1/\alpha} + 2^{-1}a_{n}^{1/\alpha} }.
\end{align}
We  first bound $\lambda_{S}$.
 For $0<\alpha<1$, by the monotonicity of $\ell^p$-norm in $p$,
 \begin{equation}\label{eq: low value part}
\lambda_{S} \le  2  \eta_{n}^{2} \sum_{(i,j)\in E_{S}} |Y_{ij}|\le 2 \eta_{n}^{2}\Big( \sum_{(i,j)\in E_{S}} |Y_{ij}|^{\alpha} \Big)^{\frac{1}{\alpha}}.
\end{equation}
Using the fact $|E_S| \le |E| = |V|-1 \leq c \log n/b_{n},$ by Lemma \ref{lem: conditioned alpha power sum asymp},
\begin{align}\label{eq: small part}
\P\paren{ \lambda_{S} > 2^{-1}a_{n}^{1/\alpha} }
\le \exp\paren{-\frac{a_{n}}{2^{2\alpha}\eta_{n}^{2\alpha}} + c\log n + o(\log n)} = o(n^{-1}),
\end{align}
where the first {condition} \eqref{312} is used in the last estimate.

Next, we bound $\lambda_{L}$. Using the fact $|f_{i}f_{j}|\le 1/2$ for $i\neq j$ (recall   $\sum_{i\in V}|f_{i}|^{2}=1$), by  the monotonicity of $\ell^p$-norm again,
\begin{equation*}
\lambda_{L}  \le   \sum_{(i,j)\in E_{L}}  |Y_{ij} |\le \Big( \sum_{(i,j)\in E_{L}} |Y_{ij}|^{\alpha} \Big)^{\frac{1}{\alpha}} .
\end{equation*}
Hence, using the fact $|E_{L}|\le d\eta_{n}^{-2}$ (see \eqref{323}), by Lemma \ref{lem: conditioned alpha power sum asymp} together with the fact
\begin{equation}\label{3a77}
    (x+y)^\alpha \ge x^\alpha  + \alpha 2^{\alpha-1} x^{\alpha-1} y,  
    \quad \forall x \ge y \ge 0 \text{ and } 0<\alpha\le 1,
\end{equation}
we deduce that for some constant $c''>0,$ 
\begin{align*} 
\P\big( \lambda_{L} &> (\log n)^{1/\alpha} + 2^{-1}a_{n}^{1/\alpha} \big) \nonumber \\
&
\le ( c\log n/b_{n}) ^{\eta_{n}^{-2}} \exp\paren{- \log n - c'' (\log n)^{1-1/\alpha}a_{n}^{1/\alpha} + d\eta_{n}^{-2}b_{n} + o(\eta_{n}^{-2}b_{n})} = o(n^{-1}).
\end{align*} 
Therefore, this along with \eqref{eq: low value part}  conclude the proof. 
  
\end{proof}

\subsection{Proof of Proposition \ref{prop: upper tail upper bound}}
Given the previous preparations, we prove 
Proposition \ref{prop: upper tail upper bound}.

\begin{proof}[Proof of Proposition \ref{prop: upper tail upper bound}]
For a small constant $\kappa>0$,  setting
\begin{align}\label{eq: a_n}
     a_n \coloneqq  \begin{cases}
        (\log n)^{ \frac{\alpha+1}{2\alpha+1} + \kappa} \quad &\alpha>1, \\
        (\log n)^{ \frac{2}{\alpha+2} + \kappa} &0<\alpha\leq 1,
    \end{cases}
 \end{align}
it suffices to verify that  
\begin{align}
 \lim_{n\to \infty }    \P\paren{\lambda_{1}(X) \ge 2^{1/\alpha} K_d (\beta/2) (\log n)^{1/\alpha} + 3a_{n}^{1/\alpha} }  = 0 \qquad (\alpha>1),
\end{align}
and
\begin{align}
 \lim_{n\to \infty }    \P\paren{\lambda_{1}(X) \ge (\log n)^{1/\alpha} + 3a_{n}^{1/\alpha} }  = 0 \qquad  (0<\alpha\le 1),
\end{align}
{By monotonicity, we assume  that $\kappa>0$ is small enough so that $a_{n}\ll \log n$.}
Let us define  
 \begin{align}\label{eq: b_n}
      \quad b_n \coloneqq  \begin{cases}
         (\log n)^{\frac{\alpha}{2\alpha+1}} \quad &\alpha>1, \\
        (\log n)^{\frac{\alpha}{\alpha+2}} &0<\alpha\leq 1.
    \end{cases}
 \end{align}
Then, the sequences $\{a_n\}_{n\ge 1}$ and $\{b_n\}_{n\ge 1}$  satisfy 
\begin{align} \label{325}
    \log\log n \ll b_{n}\ll a_{n} 
    ,\quad \log\log n \ll \log n/b_{n} \ll a_{n} \ll \log n.
\end{align}
as well as 
\begin{align} \label{340}
    \begin{cases}
       a_n  \gg  b_n^{\alpha / (\alpha+1)} (\log n)^{1/(\alpha+1)} \quad &\alpha>1, \\
        a_n \gg  b_n^{\alpha/2} (\log n)^{1- (\alpha/2)} \quad &0<\alpha\leq 1.
    \end{cases}
\end{align} 
We decompose the network $X=X^{(1)}+X^{(2)}$ with
\begin{equation} \label{611}
X^{(1)}_{ij} = (A_{ij}  \1_{|W_{ij}|>b_{n}^{1/\alpha}} ) W_{ij}, \quad X^{(2)}_{ij} = (A_{ij} \1_{|W_{ij}|\le b_{n}^{1/\alpha}}  ) W_{ij}.
\end{equation}
We separately consider the cases $\alpha>1$ and $0<\alpha\leq 1.$

~

\textbf{Case 1.~$\boldsymbol{\alpha > 1}$.} 
We have
\begin{multline}\label{330vv}
\P\paren{\lambda_{1}(X) \ge 2^{1/\alpha} K_d (\beta/2) (\log n)^{1/\alpha} + 3a_{n}^{1/\alpha} } \\
\le \P\paren{\lambda_{1}(X^{(1)}) \ge 2^{1/\alpha} K_d (\beta/2) (\log n)^{1/\alpha} + 2a_{n}^{1/\alpha} }
 + \P\paren{\lambda_{1}(X^{(2)}) \ge a_{n}^{1/\alpha} }.
\end{multline}
 First,
$\lambda_{1}(X^{(2)})$ can be bounded in a straightforward way. In fact, since the largest eigenvalue of the $d$-regular graph is equal to $d$,
\begin{align} \label{x2}
\lambda_{1}(X^{(2)}) \le  b_{n}^{1/\alpha} \lambda_{1}(A) = b_{n}^{1/\alpha}d    .
\end{align}
Since $a_{n}\gg b_{n}$ (see \eqref{325}), 
\begin{equation} \label{331}
\P\paren{\lambda_{1}(X^{(2)}) \ge a_{n}^{1/\alpha} }= 0.
\end{equation}

Next, we bound  $\lambda_{1}(X^{(1)}).$  Let $\cE_{1}$ be the event  defined in Proposition \ref{prop: almost tree nbhd} with $R_n\coloneqq\lfloor 3\log n/b_{n} \rfloor$ and $w\coloneqq 1$, i.e.,
\begin{align} \label{341}
    \text{$B_{R_n}(i)$ (in the  graph $A$) has at most one tree-excess edge},\quad \forall i \in [n],
\end{align}
and
\begin{align} \label{342}
    \size{\parenn{i\in [n]:\text{$B_{R_n}(i)$ (in the  graph $A$) contains a cycle}}}\le (d-1)^{4R_n}.
\end{align}
Note that these properties  hold for any subgraph of $A$ as well (w.r.t.~graph distance induced by the subgraph), in particular for  the subgraph $A^{(1)}$.
Then, by \eqref{final} in  Proposition \ref{prop: almost tree nbhd} (which is applicable since $\log\log n\ll R_n=\lfloor 3\log n/b_{n} \rfloor \ll \log n$, see \eqref{325}), for sufficiently large $n$,
\begin{align}  \label{332}
  \P(\cE_{1}^{c}) \le n^{-1/2}.
\end{align}
Next, define the event $\cE_{2}$ by
\begin{equation} \label{ce2}
\cE_{2} \coloneqq \left\{\text{Every connected component  in $A^{(1)}$ has at most $R_n (=\lfloor 3\log n/b_{n}\rfloor)$ edges}\right\}.
\end{equation}
Since $A^{(1)}$ can be regarded as  a  (bond) percolation with a connectivity probability $p\leq C_2 e^{-b_n}$ on the random $d$-regular graph $A$  with $|V(A)| = n$,
by Lemma \ref{lem: size of connected compo}, 
\begin{align}\label{eq: component size}
  \P(\cE_{2}^{c})  = \E   [\P(\cE_{2}^{c}| A )  ] \le  n^{-1}.
\end{align}
We now extract tree-excess edges in the graph $A^{(1)}$ and then analyze their contribution.
Let $C_{1},\cdots,C_{m}$ be connected components in $A^{(1)}$. For each $k\in [m]$, we denote by $E_{\mathsf{Excess}}(k)$ the collection of tree-excess edges in  the component $C_{k}$ (the spanning tree of $C_k$ is taken  arbitrarily). 
Then, define 
\begin{equation*}
E_{\mathsf{Excess}} \coloneqq \bigcup_{k\in [m]} E_{\mathsf{Excess}}(k) \subseteq E(A^{(1)} ).
\end{equation*} 
This induces a further decomposition $X^{(1)}=X^{(1,1)}+X^{(1,2)}$ as follows:
\begin{equation}\label{eq: further decomp}
X^{(1,1)}_{ij} = X^{(1)}_{ij}\1_{ (i,j)\notin E_{\mathsf{Excess}}} , \quad X^{(1,2)}_{ij} = X^{(1)}_{ij}\1_{ (i,j)\in E_{\mathsf{Excess}}} .
\end{equation}
In other words,  $X^{(1,1)}$  and   $X^{(1,2)}$ can be regarded as networks on the random graphs 
\begin{align}\label{eq: further decomp +}
    A^{(1,1)}_{ij} = A^{(1)}_{ij}\1_{ (i,j)\notin E_{\mathsf{Excess}}},\quad A^{(1,2)}_{ij} = A^{(1)}_{ij}\1_{ (i,j)\in E_{\mathsf{Excess}}}
\end{align}
with edge-weights given by i.i.d.~Weibull distributions  with a shape parameter $\alpha,$ conditioned to be  greater than $b_{n}^{1/\alpha} $ in absolute value.  Then, we write 
\begin{multline}\label{eq: upper tail light 0}
\P\paren{\lambda_{1}(X^{(1)}) \ge 2^{1/\alpha} K_d (\beta/2) (\log n)^{1/\alpha} + 2a_{n}^{1/\alpha} } \\
\le   \P\paren{\lambda_{1}(X^{(1,1)}) \ge 2^{1/\alpha} K_d (\beta/2) (\log n)^{1/\alpha} + a_{n}^{1/\alpha} } +\P\paren{\lambda_{1}(X^{(1,2)}) \ge a_{n}^{1/\alpha} }.
\end{multline}

We claim that under the event $\cE_{1}\cap\cE_{2}$,
\begin{align} \label{315}
    \text{$A^{(1,2)}$ consists of at most $(d-1)^{4{R_n}}$ vertex-disjoint edges}
\end{align} 
and
\begin{align} \label{316}
    \text{every connected component in  $A^{(1,1)}$  is tree}.
\end{align} 
To deduce \eqref{315}, first notice that since the number of edges in every  component $C_{k}$ of $A^{(1)}$ is at most $R_n$ (see \eqref{ce2}), the diameter of $C_k$ is at most $R_n.$ Thus by \eqref{341}, $C_k$ has at most one tree-excess edge {(recall that properties \eqref{341} and \eqref{342} hold for the subgraph $A^{(1)}$ of $A$ as well)}. Also, by \eqref{342}, the number of components  having an tree-excess edge is at most  $(d-1)^{4{R_n}}$. Hence, we obtain \eqref{315}. In addition, \eqref{316} follows from the observation that all tree-excess edges in every  component of $A^{(1)}$ are classified into $A^{(1,2)}$.

~

Now, we deduce that  $\lambda_{1}(X^{(1,2)})$ is  negligible. By Lemma \ref{lem: conditioned alpha power sum asymp} (with $m=1$) and a union bound together with \eqref{315}, for any $x>1$, under the event $\cE_1 \cap \cE_2,$
\begin{align} \label{345}
\P\paren{ \lambda_{1}(X^{(1,2)}) \ge x \mid A^{(1,2)} }
\le (d-1)^{4R_n} \exp( - x^\alpha + b_{n} + O(\log x)).
\end{align}
Setting $x:=a_n^{1/\alpha}$ and  recalling $R_n=\lfloor 3\log n/b_{n} \rfloor$, using \eqref{332} and  \eqref{eq: component size},
\begin{equation}\label{eq: upper tail light 1}
\P\paren{\lambda_{1}(X^{(1,2)}) \ge a_{n}^{1/\alpha} }
\le  { (d-1)^{4\lfloor 3\log n/b_{n} \rfloor} \exp( - a_{n} + b_{n} + o(b_{n}) )
+ o(1) = o(1)},
\end{equation}
where the last estimate follows from the facts $a_{n}\gg \log n/b_{n}$ and $a_n \gg b_n\gg\log\log n$ (see  \eqref{325}).

Next, we bound $\lambda_{1}(X^{(1,1)})$. {Let $T_{1},\cdots,T_{m}$ be connected components in  $X^{(1,1)}$.}  Since
\begin{equation*}
\lambda_{1}(X^{(1,1)}) = \max_{1\le k\le m} \lambda_{1}({T_{k}}),
\end{equation*}
we aim to bound $\lambda_1(T_k)$ for each component $T_{k}$.

Note that under the event $\cE_{2},$  $|V(T_{k})| = |E(T_{k})| +1 \leq \lfloor 3\log n/b_{n} \rfloor +1$. Hence, since $T_k$ is tree  (see \eqref{316}) whose maximum degree is at most $d$,  by Proposition \ref{keyprop} (which is applicable by the condition \eqref{340}), under the event $\cE_{1}\cap\cE_{2}$, 
\begin{align*}
      \P\paren{ \lambda_{1}(T_{k}) \ge 2^{1/\alpha} K_d (\beta/2) (\log n)^{1/\alpha} + a_{n}^{1/\alpha} \mid A^{(1,1)} } = o(n^{-1}).
 \end{align*}
Therefore,
by a union bound over components $T_1,\cdots,T_m$ (note that $m\le n$) together with  \eqref{332}  and \eqref{eq: component size},
\begin{align}  \label{333}
     \P \big( \lambda_{1}(X^{(1,1)}) \ge 2^{1/\alpha} K_d (\beta/2) (\log n)^{1/\alpha} + a_{n}^{1/\alpha} \big) = n \cdot o(n^{-1}) + o(1) = o(1).
\end{align} 
Therefore, by \eqref{331}, \eqref{eq: upper tail light 1} and \eqref{333}, we conclude the proof.

~

\textbf{Case 2.~$\boldsymbol{0 < \alpha \le 1}$.}
We have
\begin{equation*}
\P\paren{\lambda_{1}(X) \ge (\log n)^{1/\alpha} + 3a_{n}^{1/\alpha} }
\le \P\paren{\lambda_{1}(X^{(1)}) \ge (\log n)^{1/\alpha} + 2a_{n}^{1/\alpha} }
+ \P\paren{\lambda_{1}(X^{(2)}) \ge a_{n}^{1/\alpha} }.
\end{equation*} 
Since $\lambda_1(X^{(2)})$ can be  bounded as in \eqref{331}, it is enough to bound $\lambda_1 (X^{(1)}).$

Let us further decompose $X^{(1)}=X^{(1,1)}+X^{(1,2)}$ as  in \eqref{eq: further decomp}, and then  bound
\begin{equation*}
\P\paren{\lambda_{1}(X^{(1)}) \ge (\log n)^{1/\alpha} + 2a_{n}^{1/\alpha} }
\le \P\paren{\lambda_{1}(X^{(1,2)}) \ge a_{n}^{1/\alpha} }
 + \P\paren{\lambda_{1}(X^{(1,1)}) \ge (\log n)^{1/\alpha} + a_{n}^{1/\alpha} }.
\end{equation*}
By the same reasoning as before, we bound $\lambda_{1}(X^{(1,2)}) $ as in \eqref{eq: upper tail light 1}. To bound $\lambda_{1}(X^{(1,1)}),$ 
let $T_{1},\cdots,T_{m}$ be  components in $X^{(1,1)}$. 
By Proposition \ref{keyprop}  (which is applicable by   the condition \eqref{340}), under the event $\cE_{1}\cap\cE_{2}$,
\begin{equation*}
\P\paren{ \lambda_{1}(T_{k}) \ge (\log n)^{1/\alpha} + a_{n}^{1/\alpha} \mid A^{(1,1)} } = o(n^{-1}).
\end{equation*} 
By a union bound together with \eqref{332} and \eqref{eq: component size},
\begin{equation} \label{336}
\P\paren{\lambda_{1}(X^{(1,1)})  \ge  (\log n)^{1/\alpha}+ a_{n}^{1/\alpha} } = n\cdot o(n^{-1}) + o(1) = o(1),
\end{equation}
Therefore, combining the above displays, we conclude the proof.

\end{proof}

Now, we establish Theorem \ref{thm: LLN}.

\begin{proof}[Proof of Theorem \ref{thm: LLN}]

{Recalling the relation \eqref{relation} and Lemma \ref{formula k},}
the lower bounds in
 \eqref{640} and \eqref{641} along with the upper bound in Proposition \ref{prop: upper tail upper bound}  immediately conclude the proof. 
\end{proof}

\section{Localization of the top eigenvector}\label{sec4}

In this section, we establish that the top eigenvector is localized with high probability. Before proceeding with the proof, we provide some main steps. Without loss of generality, we assume that the top eigenvector $\mathbf{f}=(f_{i})_{i \in [n]}$ has a unit $\ell^2$-norm, i.e.~$\norm{\mathbf{f}}_2=1.$
\begin{enumerate}
    \item  Let us first consider the case $\alpha>2$. As in the proof of Proposition \ref{prop: upper tail upper bound}, 
 we decompose  $X = X^{(1,1)}+X^{(1,2)}+X^{(2)}$, where  $X^{(1,2)}$ and $X^{(2)}$ are spectrally negligible and $X^{(1,1)}$ consists of vertex-disjoint trees. 
 By the lower bound on $\lambda_1(X)$,  for any $\e>0$, the contribution from the network $X^{(1,1)}$ is at least $(1-\e) h_{d}(\alpha)(\log n)^{1/\alpha}.$

 \item For each tree-component {$T_k$}  in $X^{(1,1)}$, define $x_k:= \sum_{i \in V(T_k)} |f_i|^2$ and $S_k$ to be the $\ell^\alpha$-norm of edge-weights in $T_k$. By H\"older inequality, the contribution from $T_k$ is bounded by
 \begin{align}\label{jl: gap} 
     \Big(  \sum_{(i,j)\in \overrightarrow{E\;}(T_k)} |f_{i}f_{j}|^{\beta} \Big)^{\frac{1}{\beta}}  \Big( \sum_{(i,j)\in \overrightarrow{E\;}(T_k)} |W_{ij}|^{\alpha} \Big)^{\frac{1}{\alpha}} 
     \le K_d(\beta/2)x_k\cdot  2^{1/\alpha}  S_k  \overset{\eqref{relation}}{=}  h_d(\alpha) x_k S_k
 \end{align}
 ($\beta$ is {the H\"older} conjugate of $\alpha$).
If some component $T_k$, whose $\ell^\alpha$-norm $S_k$ is at most $(1-\e^{1/2})(\log n)^{1/\alpha}$, has a non-negligible value $x_k$, then the contribution from the network $X^{(1,1)}$ cannot be larger than the desired value mentioned in Step (1). Hence, $\ell^2$-mass of top eigenvector is essentially concentrated on the components whose $S_k$-value is at least $(1-\e^{1/2})(\log n)^{1/\alpha}$. By a tail estimate of $S_k$, the number of  such components is at most $n^{\alpha \e^{1/2}}$.
 Since the size of every component in $X^{(1,1)}$ is relatively small, we establish the localization on $n^{\alpha \e^{1/2} + o(1)}$ vertices.

 \item If edge-weights possess a heavier tail (i.e.~$0<\alpha<2$), then one can refine  the above arguments by working on the smaller fluctuation scale.  
 For $a_n  \ll \log n$ defined in \eqref{eq: a_n}, with the aid of  a sharper lower bound \eqref{620}, we show that   the contribution from $X^{(1,1)}$ is at least $  (\log n)^{1/\alpha} - \e a_n^{1/\alpha}$.  
 By the same reasoning as above, the top eigenvector is essentially concentrated on the components whose $S_k$-value is at least $ (\log n)^{1/\alpha} - \e^{1/2}a_n^{1/\alpha}$, {and the number of such such components is at most $\exp\big(\alpha\ep^{1/2} (\log n)^{\zeta_{\alpha}} \big)$ for some $\zeta_{\alpha}\in(0,1)$.
 } 

 \item  There is another scenario which yields a gap in the inequality \eqref{jl: gap}. Let $F_k$ be the first factor in LHS of \eqref{jl: gap}. The existence of a component $T_{k}$, whose $F_k$-value is at most $(1-\ep)K_{d}(\beta/2) x_{k}$, with a non-negligible value $x_{k}$ 
 contradicts Step (1).  
By the structure of the ``near''-maximizer of the variational problem \eqref{general} (see Lemma \ref{lemma 4.1}), in each component $T_k$, the top eigenvector is localized on the single edge.
 
\end{enumerate}

\begin{proof} [Proof of Theorem \ref{thm: localization}]

We use the same notations as in the proof of Proposition \ref{prop: upper tail upper bound}. Recall that  for a  small enough constant $\kappa>0,$ 
\begin{align}  \label{610}
     a_n \coloneqq  \begin{cases}
        (\log n)^{\frac{\alpha+1}{2\alpha+1} + \kappa} \quad &\alpha>1, \\
        (\log n)^{\frac{2}{\alpha+2} + \kappa} &0<\alpha\leq 1
    \end{cases}
 \ \text{ and }  \quad b_n \coloneqq  \begin{cases}
         (\log n)^{\frac{\alpha}{2\alpha+1}} \quad &\alpha>1, \\
        (\log n)^{\frac{\alpha}{\alpha+2}} &0<\alpha\leq 1
    \end{cases}
 \end{align} 
 ($b_n$ denotes the  truncation parameter of edge-weights, see \eqref{611}).
 Then, these sequences satisfy the relations \eqref{325} and \eqref{340}. 
 
~

We first recall the lower bound on $\lambda_1(X).$ Let $\e>0$ be a small constant.
By \eqref{640},  
for $\alpha>2$, 
 \begin{equation} \label{524--}
          \P(\lambda_1(X) \ge (1 - \ep)h_d(\alpha)(\log n)^{1/\alpha}) \ge 1-{\exp( - n^{\ep/4})},
 \end{equation}
and in the case $0<\alpha \le 2,$ {as in \eqref{620} (with the fluctuation $(\log n)^{\tau/\alpha}$ replaced by $\e a_n^{1/\alpha}$),}
 \begin{align}\label{eq: lower bound heavy}
     \P(\lambda_1(X) \ge  (\log n)^{1/\alpha} - \e a_n^{1/\alpha}) \ge 1- {\exp( - n^{   \ep \min(\alpha,1) \cdot a_{n}^{1/\alpha}(\log n)^{-1/\alpha}+ o(1)}  )} \ge 1-\exp(-\e^\alpha a_n),
 \end{align}
 where we recall the quantity $a_n$ in  \eqref{610}, in the last inequality.

 ~

To  control $\lambda_1(X)$ from above, let $\mathbf{f}=(f_{i})_{i \in [n]}$ be a (unit) top eigenvector of $X$ and write
\begin{align} \label{600}
    \lambda_1(X) = \sum_{(i,j)\in\overrightarrow{E}(A) }X_{ij}f_{i}f_{j} = \sum_{(i,j)\in\overrightarrow{E}(A)}X^{(1,1)}_{ij}f_{i}f_{j}+\sum_{(i,j)\in\overrightarrow{E}(A)}X^{(1,2)}_{ij}f_{i}f_{j}+\sum_{(i,j)\in\overrightarrow{E}(A)}X^{(2)}_{ij}f_{i}f_{j} 
\end{align}
(recall \eqref{611} and \eqref{eq: further decomp} for the decomposition).
By a variational formula together with \eqref{x2} and the first condition in \eqref{325}, 
\begin{equation}\label{51200}
    \sum_{(i,j)\in\overrightarrow{E}(A) }X^{(2)}_{ij}f_{i}f_{j} \le \lambda_{1}(X^{(2)}) \le b_{n}^{1/\alpha}d \ll \e a_{n}^{1/\alpha}.
\end{equation} 
Also, by \eqref{345} (with $x\coloneqq \e a_n^{1/\alpha}$, recall $R_n=\lfloor 3\log n/b_{n} \rfloor$) along with \eqref{332} and \eqref{eq: component size},  
\begin{align} \label{512--}
\P\Big(\sum_{(i,j)\in\overrightarrow{E} (A)} & X^{(1,2)}_{ij}f_{i}f_{j}  \ge   \e   a_{n}^{1/\alpha} \Big) \le  \P\big(  \lambda_1(X^{(1,2)}) \ge   \e   a_{n}^{1/\alpha} \big) \nonumber \\
&\le    (d-1)^{4\lfloor 3\log n/b_{n} \rfloor}  \exp( -\e^\alpha  a_{n} + b_{n} +  O(\log a_n) ) + 2n^{-1/2} \le \exp( -\e^\alpha  a_{n} /2 ),
    \end{align}
where we used $\log n \gg a_{n}\gg b_{n}$ and $a_{n}\gg \log n/ b_{n}$ (see  \eqref{325}) in the last inequality.

~
 
To analyze  $X^{(1,1)}$, consisting of vertex-disjoint trees
  $T_1,\cdots,T_m$,  for each $k\in [m]$, define
\begin{equation}\label{eq: def of x_k}
x_{k} \coloneqq \sum_{i\in V(T_{k})}|f_{i}|^{2} . 
\end{equation}
Since $\norm{\textup{\textbf{f}}}_2=1,$ we have  $\sum_{k=1}^m x_k = 1.$ As in the proof of Proposition \ref{keyprop}, let $\eta_{n}>0$ be a truncation level  satisfying  \eqref{312}:
\begin{align} \label{cond}
    a_{n}\eta_{n}^{-\min\{2,2\alpha\}} \gg\log n,\quad  (\log n)^{1-1/\alpha}a_{n}^{1/\alpha}\gg \eta_{n}^{-2}b_n
\end{align}
(such $\eta_n>0$ exists due to the  condition \eqref{340}).
For each $k\in [m]$, define
\begin{align} \label{eq: I_k}
    I_k \coloneqq \{ i\in V(T_k) : |f_{i}|^{2}<\eta_{n}^{2}x_{k}\}.
\end{align} 
Then,   by \eqref{eq: def of x_k},
\begin{align} \label{528--}
     |V(T_k) \setminus I_k| \le \eta_{n}^{-2}.
\end{align} 
We write 
    \begin{align} \label{525--}
\sum_{(i,j)\in\overrightarrow{E}(A)}X^{(1,1)}_{ij}f_{i}f_{j} & = \sum_{k=1}^{m}\sum_{\substack{(i,j)\in \overrightarrow{E}(T_{k})\\ i,j\in I_k}}X^{(1,1)}_{ij}f_{i}f_{j}  +\sum_{k=1}^{m}\sum_{\substack{(i,j)\in \overrightarrow{E}(T_{k})\\ i\text{ or }j\notin I_k}}X^{(1,1)}_{ij}f_{i}f_{j} 
 \nonumber \\
 &=: \sum_{k=1}^{m} \lambda_k^{(1)}+ \sum_{k=1}^{m} \lambda_k^{(2)}. 
    \end{align}
     
\textbf{Step 1. Proof of $\ell^2$-localization \eqref{2} for $\alpha>1$.} As in the proof of Proposition \ref{keyprop}, the first term above is shown to be negligible. To see this,
 let  $\beta>1$ be the  H\"older conjugate of $\alpha$. Then, 
similarly as in \eqref{334},  for each $k\in [m],$
\begin{align*}
    2\sum_{\substack{(i,j)\in E(T_{k})\\ i,j\in I_k}} |f_{i}f_{j}|^{\beta}
\le   \eta_{n}^{2(\beta-1)}x_{k}^{\beta-1} \sum_{(i,j)\in E(T_{k})} (|f_{i}|^{2}+|f_{j}|^{2}) \leq d \eta_{n}^{2(\beta-1)} x_k^{\beta}.
\end{align*}
 Under the event $\cE_2$ (recall its definition in   \eqref{ce2}),
 \begin{align} \label{111}
     |E(T_k)| \le \frac{3\log n}{b_n} \ll a_n \eta_{n}^{-2}  
 \end{align} 
 (see the first condition in \eqref{cond}). Hence,
 by H\"older inequality together with Lemma \ref{lem: conditioned alpha power sum asymp},
\begin{align*}
    \P\big( \lambda_k^{(1)}  \ge \e  x_{k} a_{n}^{1/\alpha} \mid A^{(1,1)}  \big) 
&\le  \P\Big({\sum_{
(i,j)\in E(T_{k})
} |X^{(1,1)}_{ij}|^\alpha}  \ge  2^{-1} \e^\alpha a_n \cdot  d^{-\alpha/\beta}   \eta_{n}^{-2} \mid A^{(1,1)}  \Big)  \\
&\le \exp\big(-2^{-1} \e^\alpha d^{-\alpha/\beta}  a_{n}\eta_{n}^{-2} + O(\log n) \big).
\end{align*}
Using the fact $\sum_{k=1}^mx_k=1$ together with a union bound, using  \eqref{eq: component size} and recalling  $a_{n}\eta_{n}^{-2} \gg \log n$, 
    \begin{align} \label{527--}
    	\P\Big(\sum_{k=1}^{m} \lambda_k^{(1)}  \ge \e a_{n}^{1/\alpha}   
     \Big) \le 2n^{-1}.  
    \end{align}
Next, we bound the second term in \eqref{525--}. Let
\begin{equation} \label{40}
S_{k}\coloneqq
\Big(\sum_{\substack{(i,j)\in E(T_{k})\\ i\text{ or }j\notin I_k}}|X^{(1,1)}_{ij}|^{\alpha}\Big)^{\frac{1}{\alpha}},\quad  F_k\coloneqq \Big(\sum_{\substack{(i,j)\in \overrightarrow{E}(T_{k})\\ i\text{ or }j\notin I_k}}|f_{i}f_{j}|^{\beta}\Big)^{\frac{1}{\beta}}
\end{equation}
(note that in the definition of $ F_k,$ the summation is taken over directed edges, this will be convenient when relating to the quantity $K_d$ in \eqref{k}).  
Then, by H\"older inequality,
\begin{equation} \label{601}
     \lambda_k^{(2)} = \sum_{\substack{(i,j)\in \overrightarrow{E}(T_{k})\\ i\text{ or }j\notin I_k}}X^{(1,1)}_{ij}f_{i}f_{j}  \le   2^{1/\alpha} S_kF_k.
\end{equation} 
Also, since $T_k$ can be embedded into the infinite $d$-regular tree, by the definition of $K_d$ in \eqref{k}, 
\begin{align} \label{665}
    F_k\le K_d(\beta/2)x_k
\end{align}
(the quantity $x_k$ arises from the identity $\sum_{i\in V(T_{k})}|f_{i}|^{2}=x_k$). 

~

Setting the quantity $\widetilde{a}_{n}$ by
\begin{equation} \label{622}
    \widetilde{a}_{n}\coloneqq
    \begin{cases*}
     a_{n} \quad & $1<\alpha\le 2$,\\
   \log n \quad & $\alpha > 2$,
    \end{cases*}
\end{equation}
we decompose the collection of components according to the values of $S_k$ and $F_k$:
\begin{align*} 
\mathcal{K}_{1,1} &\coloneqq \{k\in[m]: S_{k} > (\log n)^{1/\alpha} -  \e^{1/2} \widetilde{a}_{n}^{1/\alpha}, \  F_k > \big(1-\ep^{1/2} \big) K_d(\beta/2)x_{k}\}, \\ 
\mathcal{K}_{1,2} &\coloneqq \{k\in[m]: S_{k} > (\log n)^{1/\alpha} - \e^{1/2} \widetilde{a}_{n}^{1/\alpha},  \ F_k \le (1-\ep^{1/2})K_d(\beta/2)x_{k}\}, \\\mathcal{K}_2 &\coloneqq \{k\in[m]: S_{k} \le (\log n)^{1/\alpha} - \e^{1/2} \widetilde{a}_{n}^{1/\alpha}\}. 
\end{align*}
Let us write
\begin{equation} \label{517vvv}
x_{\mathcal{K}_{1,1}} \coloneqq \sum_{k\in \mathcal{K}_{1,1}}x_{k} ,\quad x_{\mathcal{K}_{1,2}} \coloneqq \sum_{k\in \mathcal{K}_{1,2}}x_{k},\quad   x_{\mathcal{K}_2} \coloneqq \sum_{k\in \mathcal{K}_2}x_{k}.
\end{equation}
From now on, to alleviate notations, we extend the function $h_d$ in Theorem \ref{thm: LLN} to 
 $(1,\infty)$ as follows: $h_d(\alpha) \coloneqq 1$ for all $\alpha \in (1,2].$ Then, by \eqref{relation} (for $\alpha>2$) and Lemma \ref{formula k},
\begin{align*}
    h_d(\alpha) = 2^{1/\alpha} K_d(\beta/2),\quad \forall \alpha>1.
\end{align*}
This together with  \eqref{601}  and  \eqref{665} imply that
\begin{align*}
    \lambda_k^{(2)} \le  
    \begin{cases}
        h_d(\alpha) x_kS_k\quad &k\in \mathcal{K}_{1,1}\cup \mathcal{K}_{2},\\ 
         (1-\ep^{1/2}) h_d(\alpha) x_kS_k\quad  &k\in \mathcal{K}_{1,2}.
    \end{cases}
\end{align*}
Hence, under the event   $\{\max_{1\le k\le m}S_{k} < (\log n)^{1/\alpha} + \e a_{n}^{1/\alpha} \},$ 
\begin{multline} \label{530}
      \sum_{k=1}^m \lambda_k^{(2)}
    \le h_d(\alpha)\big[ ((\log n)^{1/\alpha} + \ep a_{n}^{1/\alpha})x_{\mathcal{K}_{1,1}} 
    + (1-\ep^{1/2})((\log n)^{1/\alpha} + \ep a_{n}^{1/\alpha})x_{\mathcal{K}_{1,2}} \\
    + ((\log n)^{1/\alpha} - \e^{1/2}  \widetilde{a}_{n}^{1/\alpha} )x_{\mathcal{K}_{2}}\big].
\end{multline}
We verify that  under the event $\cE_2$,
    \begin{align} \label{520--}
    	\P\big( \max_{1\le k\le m}S_{k} \ge (\log n)^{1/\alpha} + \e a_{n}^{1/\alpha} \mid A^{(1,1)} \big)  \le { \exp( -  \ep \alpha (\log n)^{1-1/\alpha}a_{n}^{1/\alpha} /2 )}.
    \end{align}
Since the maximum degree of $T_k$ is at most $d$, by \eqref{528--}, for each $k\in [m],$ 
\begin{align} \label{556}
     | \{(i,j)\in E(T_k): i\text{ or }j\notin I_k \}| \le d \eta_{n}^{-2}.
\end{align}
   Hence, by Lemma \ref{lem: conditioned alpha power sum asymp}  together with a union bound over all possible subsets $I_k^c$ in $V(T_k)$ and the inequality $(x+y)^\alpha \ge x^\alpha + \alpha x^{\alpha-1} y$, under the event $\cE_2$ (note that $|V(T_k)| = |E(T_k)|+1 \le 3\log n / b_n+1$, see \eqref{111}), 
\begin{align} \label{666}
    \P\Big( S_{k} & \ge (\log n)^{1/\alpha} + \e a_{n}^{1/\alpha}  \mid A^{(1,1)}  \Big)   \nonumber \\
&     \le  { \Big( \frac{3\log n}{b_{n}} +1 \Big)^{\eta_{n}^{-2}} \cdot } \exp\big( - \log n - \e \alpha (\log n)^{1-1/\alpha}a_{n}^{1/\alpha} + d \eta_{n}^{-2}b_{n} + O( \eta_{n}^{-2} \log \log n) \big)
\end{align}
(Lemma \ref{lem: conditioned alpha power sum asymp}  is applicable since $\log n \gg \eta_n^{-2}$, see \eqref{cond} and recall $a_n \ll \log n$). By \eqref{entropy} and    recalling $(\log n)^{1-1/\alpha}a_{n}^{1/\alpha} \gg  \eta_{n}^{-2}b_n$ (see \eqref{cond}), the above quantity is bounded by $$ \exp\big( - \log n - \e \alpha (\log n)^{1-1/\alpha}a_{n}^{1/\alpha} /2\big).$$
By a union bound over all components $T_1,\cdots,T_m$ (note that $m\le n$), we obtain \eqref{520--}.

~

Therefore,   by 
 \eqref{51200}, \eqref{512--}, \eqref{527--} and \eqref{530}, with probability at least $1-\exp(-\e^\alpha a_n/3)$ (recall $a_n \ll \log n$),
\begin{multline} \label{605}  
      \lambda_1(X)
    \le h_d(\alpha)\big[ ((\log n)^{1/\alpha} + \ep a_{n}^{1/\alpha})x_{\mathcal{K}_{1,1}} 
    + (1-\ep^{1/2})((\log n)^{1/\alpha} + \ep a_{n}^{1/\alpha})x_{\mathcal{K}_{1,2}} \\
    + ((\log n)^{1/\alpha} -  \e^{1/2}\widetilde{a}_{n}^{1/\alpha})x_{\mathcal{K}_{2}}\big] + 3 \e a_{n}^{1/\alpha} .
\end{multline} 
We combine this with the lower bound \eqref{524--} (for $\alpha>2$) and \eqref{eq: lower bound heavy} (for $1<\alpha\le 2$) to show that   $x_{\mathcal{K}_{1,1}}$ is close to 1. Indeed,  when $\alpha>2$, recalling the quantity $\widetilde{a}_{n}$ in \eqref{622} and using $x_{\mathcal{K}_{1,1}}=1-x_{\mathcal{K}_{1,2}}-x_{\mathcal{K}_{2}}$, 
with probability at least $1-\exp(-\e^\alpha a_n/4)$,
\begin{align*}
    - \ep (\log n)^{1/\alpha} \le C \ep a_{n}^{1/\alpha}  - (\ep^{1/2}(\log n)^{1/\alpha} +\ep^{3/2}a_{n}^{1/\alpha}) x_{\mathcal{K}_{1,2}}
    - (\ep a_{n}^{1/\alpha}+  \ep^{1/2}(\log n)^{1/\alpha}) x_{\mathcal{K}_{2}} .
\end{align*}
Using $a_n \ll \log n$, this implies that for {sufficiently large $n,$} 
\begin{equation}\label{conclusion44}
    x_{\mathcal{K}_{1,2}} + x_{\mathcal{K}_{2}} \le 10\e^{1/2}.
\end{equation}
Also in the case $1<\alpha\le 2$, by similar calculations, we deduce the above bound as well.  
Hence, 
\begin{align}\label{516}
\P(x_{\mathcal{K}_{1,1}} \ge 1 - 10\e^{1/2}) \ge
 1-\exp(-\e^\alpha a_n/4) \overset{\eqref{610}}{\ge}    1-\exp(-(\log n)^{1/2})/2,  
\end{align}
i.e.~$\ell^2$-mass of the top eigenvector is essentially concentrated on the components in $\mathcal{K}_{1,1}.$

~

Next, we establish that $|\mathcal{K}_{1,1}|$ is relatively small.
Note that although edge-weights are conditionally independent given $A^{(1,1)}$, $S_k$s may \emph{not} be since $I_k$s are random subsets of $V(T_k)$. To detour this problem, observe that conditionally on $A^{(1,1)}$, 
 for any collection of \emph{deterministic} subsets  $J_{\ell_1}\subseteq V(T_{\ell_1}),\cdots, J_{\ell_t}\subseteq V(T_{\ell_t})$ ($1\le t\le m$, $ 1\le \ell_1<\cdots<\ell_t\le m$) such that $|J_{\ell_k}|\le\eta_{n}^{-2}$ for $k=1,\cdots,t$, by  Lemma \ref{lem: conditioned alpha power sum asymp},
\begin{multline*}
    \P \Big( \Big(\sum_{\substack{(i,j)\in E(T_{\ell_k}),  \  i\text{ or }j\in J_{\ell_k}}}|X^{(1,1)}_{ij}|^{\alpha}\Big)^{\frac{1}{\alpha}}
    > (\log n)^{1/\alpha} -  \e^{1/2}  \widetilde{a}_{n}^{1/\alpha}, \quad \forall k=1,\cdots,t  \mid A^{(1,1)} \Big) \\
    \le \Big(\exp\big( -\log n + \alpha (\log n)^{1-1/\alpha} \e^{1/2}  \widetilde{a}_{n}^{1/\alpha} + d\eta_{n}^{-2}b_{n} + O( \eta_{n}^{-2} \log \log n) \big)\Big)^{t},
\end{multline*}
where we used $| \{(i,j)\in E(T_{\ell_{k}}): i\text{ or }j \in J_{\ell_{k}} \}| \le d \eta_{n}^{-2}$ (similarly as in \eqref{556})
and the inequality $(x-y)^\alpha \ge x^\alpha - \alpha x^{\alpha-1} y$ 
 ($\forall x\ge y \ge  0$ and $\alpha>1$).
Hence, using the fact $\eta_{n}^{-2}b_n\ll (\log n)^{1-1/\alpha}\widetilde{a}_{n}^{1/\alpha} $ (see \eqref{cond} and recall $\widetilde{a}_{n} \ge a_n$), by a union bound over all such possible collections $\{J_{\ell_k}\}_{k=1,\cdots,t}$, under the event $\cE_2$, 
\begin{multline*} 
    \P(|\mathcal{K}_{1,1}|\ge t  \mid A^{(1,1)}  )
    \le 
    {\binom{m}{t}} \cdot \Big( \frac{3\log n}{b_{n}}  +1 \Big)^{\eta_{n}^{-2} t} \cdot 
    \Big(\exp\big( -\log n +1.5  \e^{1/2}  \alpha (\log n)^{1-1/\alpha}  \widetilde{a}_{n}^{1/\alpha}   \big) \Big)^{t}.
\end{multline*}
By \eqref{entropy} along with the fact $t ! \ge (t/e)^t$ and $m\le n,$ for sufficiently large $n$,
\begin{equation} \label{604}
	\P(|\mathcal{K}_{1,1}|\ge t   \mid A^{(1,1)} ) \le (e/t)^t \exp( 2\alpha \e^{1/2}   (\log n)^{1-1/\alpha} \widetilde{a}_{n}^{1/\alpha} t ).
\end{equation}
In particular, setting $t \coloneqq \lfloor \exp( 3\alpha \e^{1/2}(\log n)^{1-1/\alpha} \widetilde{a}_{n}^{1/\alpha} ) \rfloor$,
\begin{align} \label{606}
    \P(|\mathcal{K}_{1,1}|\ge \lfloor \exp( 3\alpha \e^{1/2}(\log n)^{1-1/\alpha} \widetilde{a}_{n}^{1/\alpha} ) \rfloor   \mid A^{(1,1)} ) \le   \exp(-e^{3\alpha \e^{1/2}   (\log n)^{1-1/\alpha} \widetilde{a}_{n}^{1/\alpha}} ).
\end{align} 
Since $|V(T_k)|  = |E(T_k)| +1\le \lfloor3\log n / b_{n}\rfloor+1$  under the event $\cE_{2}$ (see \eqref{111}), by \eqref{eq: component size},
with probability at least $1-2n^{-1}$,
\begin{align} \label{612}
|\cup_{k\in \mathcal{K}_{1,1} } V(T_k)|\le 
    \begin{cases}
         \exp(4\alpha \e^{1/2} \log n)& \alpha > 2,\\
          \exp(4\alpha \e^{1/2} (\log n)^{ \frac{2\alpha}{2\alpha+1} + \frac{\kappa}{\alpha} } ) & 1<\alpha\le 2.
    \end{cases}
\end{align}
Therefore, by taking a small  $\kappa>0$ and adjusting the value of $\e>0$, equipped with \eqref{516} and \eqref{606}, we establish \eqref{2} with  $\mathcal{I}:= \cup_{k\in \mathcal{K}_{1,1} } V(T_k)$, in the case $\alpha>1$.

 ~

\textbf{Step 2. Localization on vertex-disjoint edges for $1<\alpha<2$.} We show that  when $1<\alpha<2$, the top eigenvector $\mathbf{f}$ is essentially localized on only few number of vertex-disjoint edges. Recall that in \eqref{516}, we established that with high probability, most of the $\ell^2$-mass of $\mathbf{f}$ is concentrated on the components in $\mathcal{K}_{1,1}$:
\begin{equation}\label{abc1}
    \sum_{k\in \mathcal{K}_{1,1}} \sum_{i\in V(T_{k})} |f_{i}|^{2} =  \sum_{k\in \mathcal{K}_{1,1}} x_{k}  \ge 1- 10\e^{1/2}.
\end{equation}
For $k\in \mathcal{K}_{1,1}$,  
\begin{equation*}
    \Big(\sum_{(i,j)\in \overrightarrow{E}(T_{k})}|f_{i}f_{j}|^{\beta}\Big)^{\frac{1}{\beta}} \ge F_k \ge \big(1-\ep^{1/2} \big) K_d(\beta/2)x_{k} =   \big(1-\ep^{1/2} \big) 2^{\frac{1}{\beta}-1 } x_{k} ,
\end{equation*}
where we used Lemma \ref{formula k} together with the fact $\beta \ge 2$ in the last identity.
Hence, recalling that $T_k$ is tree and $\sum_{i\in V(T_k)} |f_i|^2 = x_k$,
by Lemma \ref{lemma 4.1} (note that our vector $\textbf{f}$ satisfies the $\ell^2$-constraint instead of the $\ell^1$-constraint), there is a constant {$c=c(\alpha)>0$} such that  for each $k\in \mathcal{K}_{1,1}$, there exists an edge $(i^{(k)},j^{(k)})\in E(T_{k})$ satisfying
\begin{equation*}
    |f_{i^{(k)}}|^{2} + |f_{j^{(k)}}|^{2} \ge (1-{c\ep^{1/4}}) x_{k}.
\end{equation*}
Combining this with \eqref{abc1},
\begin{equation*}
    \sum_{k\in \mathcal{K}_{1,1}} 
 ( |f_{i^{(k)}}|^{2} + |f_{j^{(k)}}|^{2} )  \ge (1-{c\ep^{1/4}}) \sum_{k\in \mathcal{K}_{1,1}} x_{k}  \ge  (1-{c\ep^{1/4}})(1-10\e^{1/2}) \ge 1-{(c+10)\e^{1/4}}.  
\end{equation*} 
Noting that edges  $\{(i^{(k)},j^{(k)})\}_{k\in \mathcal{K}_{1,1}}$ are vertex-disjoint, recalling the bound on $|\mathcal{K}_{1,1} |$ in \eqref{606} and the quantity $\widetilde{a}_{n}$ in \eqref{622},
  we conclude the proof by adjusting the value of $\e>0$ and taking a small $\kappa>0$.

~

\textbf{Step 3. Localization for $\mathbf{0\boldsymbol{<\alpha\le}1}$.} Since the argument is similar as in the case $\alpha>1$, except the application of H\"older  inequality  in \eqref{601}, we briefly sketch a proof.
To bound $\lambda_k^{(1)}$, similarly as in \eqref{eq: low value part} but recalling that the threshold value is $\eta_n^2x_k$  (see \eqref{eq: I_k}),
\begin{equation*}
    \lambda_k^{(1)} \le  2 \eta_{n}^{2}x_{k} \Big( \sum_{\substack{(i,j)\in {E}(T_{k})\\ i,j\in I_k}}|X^{(1,1)}_{ij}|^{\alpha} \Big)^{\frac{1}{\alpha}}.
\end{equation*}
Thus, by Lemma \ref{lem: conditioned alpha power sum asymp} and a union bound, using  the condition \eqref{cond},
there exists a  constant $c = c(\alpha,\e)>0$ such that under the event $\cE_2$,
    \begin{align} \label{527--vv}
    	\P\Big(\sum_{k=1}^{m} \lambda_k^{(1)}  > \ep a_{n}^{1/\alpha}  \mid A^{(1,1)}  \Big)
     \le n \exp\Big(- 2^{-\alpha} \ep^{\alpha} \eta_{n}^{-2\alpha} a_{n} + O(\log n)\Big) \le { \exp( - c \eta_{n}^{-2\alpha} a_{n} ) }.
    \end{align} 
To bound $\lambda_k^{(2)},$ define 
 \begin{equation*}
     M_{k} \coloneqq \max_{(i,j)\in E(T_{k})} (|f_{i}|^{2} + |f_{j}|^{2}).
 \end{equation*}
Using the fact $|f_if_j| \le M_k/2$ for any $(i,j)\in E(T_{k})$, 
\begin{align}
     \lambda_k^{(2)} = 2  \sum_{\substack{(i,j)\in E(T_{k})\\ i\text{ or }j\notin I_k}}X^{(1,1)}_{ij}f_{i}f_{j} \le M_k \sum_{\substack{(i,j)\in E(T_{k})\\ i\text{ or }j\notin I_k}} |X^{(1,1)}_{ij} |\le  M_kS_k
 \end{align}
(recall that $S_k$ is defined in  \eqref{40}), where in the last inequality we used the  monotonicity of $\ell^{p}$-norm in $p$ 
 (recall that $0<\alpha\le 1$).

Now,  define
\begin{align*}
\mathcal{K}_{1,1} &\coloneqq \{k\in[m]: S_{k} > (\log n)^{1/\alpha} - \e^{1/2} a_{n}^{1/\alpha},  \ M_{k} > (1- {\ep^{1/2}})x_{k}\}, \\ 
\mathcal{K}_{1,2} &\coloneqq \{k\in[m]: S_{k} > (\log n)^{1/\alpha} -  \e^{1/2} a_{n}^{1/\alpha},  \ M_{k} \le (1- {\ep^{1/2}})x_{k}\}, \\
 \mathcal{K}_2 &\coloneqq \{k\in[m]: S_{k} \le (\log n)^{1/\alpha} - \e^{1/2} a_{n}^{1/\alpha} \}.
\end{align*}
With the aid of {inequality \eqref{3a77}}, 
similarly as in \eqref{666}, 
   \begin{align}  
    	\P\big( \max_{1\le k\le m}S_{k} \ge (\log n)^{1/\alpha} + \e a_{n}^{1/\alpha} \mid A^{(1,1)} \big)  \le
    {\exp( - \ep \alpha (\log n)^{1-1/\alpha}a_{n}^{1/\alpha}/2 )}.
    \end{align}
Hence, as in the case $\alpha>1$, we deduce that with probability at least $1 - \exp( -\e^\alpha a_n/3)$,
\begin{multline} 
 \lambda_1(X)
 \le \big( (\log n)^{1/\alpha} + \ep a_{n}^{1/\alpha} \big)x_{\mathcal{K}_{1,1}} 
 + (1-{\ep^{1/2}})\big( (\log n)^{1/\alpha} + \ep a_{n}^{1/\alpha} \big)x_{\mathcal{K}_{1,2}} \\+ \big( (\log n)^{1/\alpha} - \e^{1/2} a_{n}^{1/\alpha} \big)x_{\mathcal{K}_2} + 3\ep a_{n}^{1/\alpha}.
\end{multline}
This along with the lower bound \eqref{eq: lower bound heavy} imply 
that  
\begin{align*}
    \P(x_{\mathcal{K}_{1,1}} \ge 1 - 10\e^{1/2}) \ge  1-\exp(-\e^\alpha a_n/4) \overset{\eqref{610}}{\ge}    1-\exp(- (\log n)^{ 1/2 })/2.
\end{align*} 
 In addition, similarly as in \eqref{606},  
 \begin{align*}
       \P(|\mathcal{K}_{1,1}|
        \ge  \lfloor \exp( 3\alpha \e^{1/2}  (\log n)^{1-1/\alpha}  a_{n}^{1/\alpha} ) \rfloor ) \le 2n^{-1}.
 \end{align*} 
Moreover, for each $k\in \mathcal{K}_{1,1}$, there exists an edge $(i^{(k)},j^{(k)}) \in E(T_k)$ satisfying $|f_{i^{(k)}}|^{2} + |f_{j^{(k)}}|^{2} > (1-{\ep^{1/2}})x_{k}$. This together with \eqref{516} implies that
\begin{equation*}
    \sum_{k\in \mathcal{K}_{1,1}} 
 ( |f_{i^{(k)}}|^{2} + |f_{j^{(k)}}|^{2} ) \ge (1- {\ep^{1/2}})(1-{10\ep^{1/2}}) \ge 1-11\e^{1/2},
\end{equation*}
concluding the proof by the arbitrariness of $\e>0$.

\end{proof}

\begin{remark}[Structure of  the top eigenvector in the case {$\alpha \ge 2$}]\label{local}
The conclusion of Step 1 in the above proof shows that the top eigenvector is essentially supported on the components $T_k$ whose $F_k$-value (see \eqref{40} for the definition) is close to {$K_d(\beta/2)x_{k}$}.  We precisely describe a structure of the top eigenvector only in the case $0<\alpha < 2$, since the exact form of the maximizer of the variational problem for $K_d(\beta/2)$ is not explicitly known when $\alpha \ge 2$.
 
\end{remark}

\section{Properties of the LLN constant}\label{sec5}
In this section, we present some properties of the function $h_d$ which describes the law of large numbers behavior of the largest eigenvalue.

\subsection{$\mathbf{h_d(\boldsymbol{\alpha})>1}$ for $\mathbf{\boldsymbol{\alpha}>2}$} \label{h_d alpha>2}
We show that the LLN constant for the (normalized) largest eigenvalue in the case of the light-tail edge-weights (i.e.~$\alpha>2$) is strictly greater than that of heavy-tail cases (i.e.~$\alpha\le 2$). This accounts for the occurrence  of the phase transition  at $\alpha=2$, i.e.~when edge-weights possess a Gaussian tail.

{We enumerate vertices in the infinite $d$-regular tree by $\{0,1,\cdots\}$ in the breadth-first search sense, i.e.~index 0 denotes the root,  indices $1,2,\cdots,d$ denote the children of $0$, and so on.} In the case $\alpha>2$, we aim to find a particular vector $\textbf{u} = (u_i)_{u=0,1,\cdots}$  (with $\ell^1$-norm equal to 1)  which makes the functional in the variational problem \eqref{h} strictly greater than 1.
Let us assume that
 $u_0=1/2$, $u_1=\cdots=u_d=1/2d$ and all other $u_i$s are zero (i.e.~supported on the levels 0 and 1 of the $d$-regular tree, in other words a ``star'' graph of degree $d$). Then, we deduce that  for any $\alpha>2,$
 \begin{align*}
    h_d(\alpha) \ge 2 \Big(d\cdot \Big(\frac{1}{2}
    \cdot \frac{1}{2d}\Big)^{\frac{\alpha}{2(\alpha-1)}}\Big)^{\frac{\alpha-1}{\alpha}} = 
 d^{\frac{\alpha-2}{2\alpha}} >1.
 \end{align*}
 Indeed,  $d^{\frac{\alpha-2}{2\alpha}}$ is the maximum  value that the variational problem \eqref{h} over vectors  $\textbf{u}$, supported on the levels 0 and 1, can take. This is because
 \begin{align*}
    2 \Big (|u_0|^{\frac{\alpha}{2(\alpha-1)}} \sum_{i=1}^d |u_i|^{\frac{\alpha}{2(\alpha-1)}}  \Big) ^{\frac{\alpha-1}{\alpha}} \le   2 ( d\cdot d^{-\frac{\alpha}{2(\alpha-1)}})^{\frac{\alpha-1}{\alpha}} |u_0|^{\frac{1}{2}}   \Big(\sum_{i=1}^d |u_i| \Big)^{\frac{1}{2}} \le  d^{\frac{\alpha-2}{2\alpha}},
 \end{align*}
where the first inequality is obtained by Jensen inequality (recall that $\alpha>2$) 
and the last inequality follows from the fact $|u_0| + \sum_{i=1}^d  |u_i| =1$.
{In fact when $\alpha>2,$ the functional in the variational problem \eqref{h} can  strictly increase once we take advantage of deeper level of the $d$-regular tree. 
We expect that if $\textbf{u} = (u_i)_{u=0,1,\cdots}$ is the maximizer, then $u_i$s take the identical value at the same level of tree.}

\subsection{$\mathbf{h_d(\boldsymbol{\alpha}) < \infty }$ for $\mathbf{\boldsymbol{\alpha}>2}$}
The finiteness of $h_d({\alpha})$ is an immediate consequence of the relation \eqref{relation} along with the inequality \eqref{112} and the discussion below it (note that the conjugate $\beta$ of $\alpha$ satisfies $\beta>1$).

\subsection{Continuity}
We establish a continuity property of the function $h_d(\alpha)$ in $\alpha$. Since the variational problem \eqref{h} is defined on the infinite-dimensional space, additional works are needed to verify the continuity of $h_d(\alpha)$. From now on, we assume that $u_i$s in \eqref{h} satisfies $u_i \ge 0$, and set 
\begin{equation}\label{set U}
	U \coloneqq \Big\{ \textbf{u} =(u_{0},u_{1},\cdots)\in[0,1]^{\mathbb{N}}: \sum_{i=0}^{\infty}u_{i} = 1  \Big\},
\end{equation}
equipped with the product topology. By Tychonoff theorem, the space $[0,1]^{\mathbb{N}}$ is compact, and since  any closed subset of compact sets is compact, $U$ is also compact.

We analyze a function $F_d: [1/2,1]\times U \rightarrow \R$  that appear  in the variational problem \eqref{h}:
\begin{equation*}
	F_d(\gamma,\textbf{u}) \coloneqq \sum_{ (i,j)\in E(\mathbb{T}_{d}) } u_{i}^{\gamma} u_{j}^{\gamma}
\end{equation*}
(the domain $[1/2,1]\times U$ is again equipped with the product topology). We show that the function $F_d$ is continuous both in the arguments $\gamma$ and $\textbf{u}$.
\begin{lemma} \label{lemma 5.1}
$F_d(\gamma,\textbf{u})$ is continuous on $[1/2,1]\times U$.
\end{lemma}
\begin{proof}
Throughout the proof, for any $\textbf{u}\in U$ and $K\in \N,$ we write $\textbf{u} = ( \textbf{u}|_{< K},\textbf{u}|_{\ge  K})$ with  
\begin{align*}
    \textbf{u}|_{< K}\coloneqq(u_0, u_{1}, \cdots,u_{K-1},0,\cdots),\quad \textbf{u}|_{\ge K}\coloneqq(0,0,\cdots,0,u_{K},u_{K+1},\cdots).
\end{align*}
	 Let $\{(\gamma_{n},\textbf{u}^{(n)})\}_{n\in \N}$ be any sequence in $[1/2,1]\times U$ such that $(\gamma_{n},\textbf{u}^{(n)})\to(\gamma,\textbf{u})$ as $n\to \infty$. Since $\norm{\textbf{u}}_1=1$, for any $\ep>0$, there is a large enough $K>0$ such that
	\begin{equation*}
		\lVert \textbf{u}|_{\ge K} \rVert_1 \le \ep.
	\end{equation*}
	 Since $\lim_{n\to \infty} \textbf{u}^{(n)} = \textbf{u}$ w.r.t.~product topology, for such $K$, for sufficiently large $n,$
\begin{align*}
    \lVert \textbf{u}^{(n)}|_{< K} - \textbf{u}|_{< K} \rVert_1 \le \e.
\end{align*}
Hence, 
 \begin{align} \label{521}
     \lVert \textbf{u}^{(n)}|_{\ge K} \rVert_1 = 1- \lVert \textbf{u}^{(n)}|_{< K} \rVert_1  \le 1- \lVert \textbf{u}|_{< K} \rVert_1 + \e  = \lVert \textbf{u}|_{\ge K} \rVert_1 + \e \le 2\ep.
 \end{align} 
Next, we write
	\begin{equation*}
		F_d(\gamma,\textbf{u}) = \Big(\sum_{ (i,j)\in E(\mathbb{T}_{d}),\min(i,j)<K  } u_{i}^{\gamma} u_{j}^{\gamma}\Big) + \Big(\sum_{ (i,j)\in E(\mathbb{T}_{d}),\min(i,j) \ge K  } u_{i}^{\gamma} u_{j}^{\gamma}\Big)
	\end{equation*}
	and
	\begin{equation*}
		F_d(\gamma_{n},\textbf{u}^{(n)}) = \Big(\sum_{ (i,j)\in E(\mathbb{T}_{d}),\min(i,j)<K  } (u^{(n)}_{i})^{\gamma_{n}} (u^{(n)}_{j})^{\gamma_{n}}\Big) + \Big(\sum_{ (i,j)\in E(\mathbb{T}_{d}),\min(i,j) \ge K  } (u^{(n)}_{i})^{\gamma_{n}} (u^{(n)}_{j})^{\gamma_{n}}\Big).
	\end{equation*}
	Since the summation in the first quantity above  only consists of finitely many terms, as $n\to \infty,$
	\begin{equation*}
 	 \sum_{ (i,j)\in E(\mathbb{T}_{d}), \min(i,j)<K  } (u^{(n)}_{i})^{\gamma_{n}} (u^{(n)}_{j})^{\gamma_{n}} 
	\rightarrow   \sum_{ (i,j)\in E(\mathbb{T}_{d}), \min(i,j)<K  } u_{i}^{\gamma} u_{j}^{\gamma}.
	\end{equation*}
	Using the fact $\gamma \ge  1/2$ and $\norm{\textbf{u}}_\infty = \max_{i\in \N} u_i \le 1$, 
	\begin{align*}
		\sum_{ (i,j)\in E(\mathbb{T}_{d}), \min(i,j) \ge K  } u_{i}^{\gamma} u_{j}^{\gamma}
	 &\le   \sum_{ (i,j)\in E(\mathbb{T}_{d}), \min(i,j)\ge K  } u_{i}^{1/2} u_{j}^{1/2}  \\
		&\le\frac{1}{2} \sum_{ (i,j)\in E(\mathbb{T}_{d}), \min(i,j)\ge K  } (u_{i}+u_j)  \le 2d\lVert \textbf{u}|_{\ge K}\rVert _1\le {2 d \ep}.
	\end{align*}
	Using \eqref{521} and the fact $\gamma_{n} \ge 1/2$, we similarly deduce that 
 \begin{align*}
     \sum_{ (i,j)\in E(\mathbb{T}_{d}),\min(i,j) \ge K  } (u^{(n)}_{i})^{\gamma_{n}} (u^{(n)}_{j})^{\gamma_{n}} \le {4 d \e}.
 \end{align*}Therefore, by the above series of information, we conclude the proof.
\end{proof}

 Let us define a function $\tilde{h}_d: [1/2,1]\to \R$ by
\begin{equation*}
	\tilde{h}_d(\gamma) \coloneqq \max_{\textbf{u}\in U} F_d(\gamma,\textbf{u}).
\end{equation*}
Since $F_d$ is continuous and its domain is compact,   $ \tilde{h}_d$ is continuous on $[1/2,1]$ as well.
Also, by the definition of $h_d$ in \eqref{h}, for any $\alpha \in (2,\infty)$, denoting by $\beta \in (1,2)$ the conjugate of $\alpha,$  
\begin{align}\label{relation2}
    h_d(\alpha) = 2\tilde{h}_d(\beta/2)^{1/\beta}.
\end{align}
Therefore, we establish the 
 continuity of $h_d$ on $(2,\infty)$.
 
 
\subsection{Limiting properties}
We establish the limiting properties of the function $h_d$, i.e.,
 $$\lim_{\alpha \downarrow 2}h_d(\alpha) = 1 \quad\text{and}\quad \lim_{\alpha \rightarrow \infty} h_d(\alpha)  =2 \sqrt{d-1}.$$

\subsubsection{Limit as $\alpha \downarrow 2$} 
We  first show that  $\lim_{\alpha \downarrow 2}h_d(\alpha) = 1 $.
By the relation \eqref{relation2}, denoting by $\beta$ the H\"{o}lder conjugate of $\alpha,$ since $\alpha \downarrow 2$ is equivalent to $\beta \uparrow 2,$ it suffices to show that
\begin{align}
  \lim_{\beta \uparrow 2}\tilde{h}_d(\beta/2)^{1/\beta}=\frac{1}{2}.
\end{align}
 Since the function $\tilde{h}_d$ is continuous on $[1/2,1]$, it reduces to show that  $\tilde{h}_d(1)= \frac{1}{4},$ in other words
\begin{align} \label{523}
    \sup_{\textbf{u}\in U} \sum_{ (i,j)\in E(\mathbb{T}_{d}) } u_{i} u_{j}=\frac{1}{4}.
\end{align}
More generally, we prove the following lemma.
\begin{lemma} \label{lem: pf of lem 2.1}
For any $\gamma\ge 1,$
    \begin{align*}
        \sup_{\textup{\textbf{u}}\in U} \sum_{ (i,j)\in E(\mathbb{T}_{d}) } u_{i}^\gamma u_{j}^\gamma =  \Big(\frac{1}{2}\Big)^\gamma \cdot \Big(\frac{1}{2}\Big)^\gamma  = 2^{-2\gamma}.
    \end{align*}
    The maximum is attained when $u_i=u_j=1/2$ for some vertices $i$ and $j$ connected by an edge.
\end{lemma}
This immediately proves  \eqref{523} and Lemma \ref{formula k}.
\begin{proof}
     Let $V_1$ (resp.~$V_2$) be the collection of vertices in the infinite $d$-regular tree whose distance from the root is odd (resp.~even). Then, there are no edges within $V_1$ and $V_2$. Hence, 
     \begin{align*}
         \sum_{ (i,j)\in E(\mathbb{T}_{d}) } u_{i}^\gamma u_{j}^\gamma\le \Big(\sum_{i \in V_1 } u_{i}^\gamma \Big )  \Big(\sum_{j \in V_2 } u_{j}^\gamma \Big )  \le  \Big(\sum_{i \in V_1 } u_{i}  \Big )^\gamma  \Big(\sum_{j \in V_2 } u_{j}  \Big )^\gamma \le  \Big(\frac{1}{2}\Big)^\gamma \Big(\frac{1}{2}\Big)^\gamma .
     \end{align*}
     Here in the second and last inequality, we used the fact $\gamma\ge 1$ and $(\sum_{i \in V_1 } u_{i}) + (\sum_{j \in V_2 } u_{j})=1$ respectively.
\end{proof}

\begin{remark} \label{remark53}
    In fact, the above proof  shows that in the case $\gamma>1$, such form of the maximizer  is unique  (modulo the choice of an edge $(i,j)$). Our proof (i.e.~availability of the 
decomposition of the vertex set of a tree into $V_1$ and $V_2$) also implies that for any tree $T$ and $\gamma \ge 1$,
  \begin{align*}
        \sup_{ \textbf{u} = (u_i)_{i\in V(T)}, u_i\ge 0,  \norm{\textbf{u}}_1=1} \sum_{ (i,j)\in E(T) } u_{i}^\gamma u_{j}^\gamma =  \Big(\frac{1}{2}\Big)^\gamma \cdot \Big(\frac{1}{2}\Big)^\gamma  = 2^{-2\gamma}.
    \end{align*}
\end{remark}
  
\subsubsection{Limit as $\alpha\rightarrow\infty$} Next, we prove that  $\lim_{\alpha \rightarrow \infty} h_d(\alpha)  =2 \sqrt{d-1} .$ By the relation \eqref{relation2}, since $\alpha \to  \infty $ is equivalent to $\beta \downarrow 1,$ it suffices to show that
\begin{align}
  \lim_{\beta \downarrow 1}\tilde{h}_d(\beta/2)^{1/\beta}= \sqrt{d-1} .
\end{align}
Since  $\tilde{h}_d$ is continuous on $[1/2,1]$, it reduces to show that  $\tilde{h}_d(1/2)= \sqrt{d-1},$ in other words
\begin{align} \label{infinite}
    \sup_{\textbf{u}\in U} \sum_{ (i,j)\in E(\mathbb{T}_{d}) } 
 \sqrt{u_iu_j}= \sqrt{d-1}.
\end{align}

Although this is a well-known fact in the context of the spectral norm of the infinite $d$-regular tree, we provide a proof for the sake of completeness. To prove the upper bound in \eqref{infinite}, 
 for $i\in V(\mathbb{T}_{d}) = \{0,1,\cdots\}$,  denote by $\text{Ch}(i)$ the set of children of $i$. For any  $\mathbf{u}=(u_{0},u_{1},\cdots)\in U$, using the inequality $2\sqrt{u_iu_j}\le \frac{1}{\sqrt{d-1}} u_i + \sqrt{d-1}u_j$ for any $(i,j)\in E$ 
 with $j\in \text{Ch}(i)$,  
\begin{align*}
\sum_{ (i,j)\in E(\mathbb{T}_{d}) }\sqrt{u_iu_j}=	\sum_{i=0}^\infty\sum_{j\in \text{Ch}(i)} \sqrt{u_iu_j}
	&\le \frac{1}{2}\sum_{i=0}^\infty \sum_{j\in \text{Ch}(i)} \Big(\frac{1}{\sqrt{d-1}} u_i + \sqrt{d-1}u_j\Big) \\
	&= \frac{d}{2\sqrt{d-1}}u_{0} + \sum_{i=1}^\infty \sqrt{d-1} u_{i} \le \sqrt{d-1}
\end{align*}
(recall that $i=0$ denotes the root of the infinite $d$-regular tree), where the last inequality follows from the condition $\sum_{i=0}^{\infty}u_{i}=1$.

Next, we prove the lower bound in \eqref{infinite}. 
For large enough $L\in\mathbb{N}$,  define  
 a vector $\mathbf{v}=(v_{0},v_{1},\cdots)$ as follows: 
\begin{equation*}
	v_{i} =
	\begin{cases}
		\frac{1}{L} & i=0 
 \ (\mathsf{root}),\\
		\frac{1}{Ld(d-1)^{\dist(0,i)-1}} & 1\le\dist(0,i)\le L-1,\\
		0 & \text{otherwise.}
	\end{cases}
\end{equation*}
Since $\sum_{i=0}^{\infty}v_{i}=1$ and $v_{i}\in[0,1]$, we have $ \textbf{v} \in U$. Moreover,  
\begin{align*}
	\sum_{ (i,j)\in E(\mathbb{T}_{d}) }\sqrt{v_iv_j} &= {\frac{\sqrt{d}}{L}} + \sum_{\dist(0,i)=1}^{L-2}\sum_{j\in \text{Ch}(i)}\sqrt{v_iv_j} \\
	& = {\frac{\sqrt{d}}{L}} + \sum_{k=1}^{L-2}d(d-1)^{k}\sqrt{\frac{1}{Ld(d-1)^{k-1}}\cdot\frac{1}{Ld(d-1)^{k}}}  \ge  {\frac{\sqrt{d}}{L}} + \frac{(L-2)\sqrt{d-1}}{L}.
\end{align*}
For any $\e>0$, this quantity is greater than $\sqrt{d-1}-\e$   for large enough $L$.

\subsection{Proof of Lemma \ref{lemma 4.1}}
Finally, we provide a proof of Lemma \ref{lemma 4.1}. 
\begin{proof}[Proof of Lemma \ref{lemma 4.1}]
 Define  \begin{equation}
        F(u) \coloneqq \Big(\sum_{(i,j)\in E} u_{i}^{\gamma}u_{j}^{\gamma}\Big)^{\frac{1}{2\gamma}} .
    \end{equation}
    Then, the condition \eqref{652}  is equivalent to  $F(u) \ge \frac{1-\e}{2}.$
    As in the proof of Lemma \ref{lem: pf of lem 2.1},  for a (fixed) vertex $v_0 \in V$, let $V_1$ (resp.~$V_2$) be the collection of vertices  whose distance from $v_0$ is odd (resp.~even). Since $\gamma>1$,
    \begin{equation} \label{41}
     \Big(\frac{1-\e}{2} \Big)^{2\gamma}\le  \sum_{ (i,j)\in E } u_{i}^{\gamma} u_{j}^{\gamma}
    \le \Big(\sum_{i \in V_1 } u_{i}^{\gamma} \Big )  \Big(\sum_{j \in V_2 } u_{j}^{\gamma} \Big )
    \le \Big(\sum_{i \in V_1 } u_{i} \Big )^{\gamma}\Big(\sum_{j \in V_2 } u_{j} \Big )^{\gamma}.
    \end{equation} 
If 
   $ | \sum_{i\in V_{1}}u_{i} - \frac{1}{2} | > \sqrt{\e} $ (equivalently, $ 
    | \sum_{j\in V_{2}}u_{j} - \frac{1}{2}| > \sqrt{\e}$), 
then  $ F(u) < (\frac{1}{2}+\sqrt{\e} )^{1/2}(\frac{1}{2}-\sqrt{\e})^{1/2}  < \frac{1-\e}{2}$, contradicting  the condition \eqref{652}.
Hence,  
\begin{equation}\label{eq: add cond}
    \frac{1}{2} - \sqrt{\e} \le \sum_{i \in V_1 } u_{i}   \le \frac{1}{2} + \sqrt{\e}
    \quad\text{and}\quad
    \frac{1}{2} -  \sqrt{\e} \le \sum_{j\in V_{2}} u_{j}  \le \frac{1}{2} + \sqrt{\e}.
\end{equation}
This in particular means  $\sum_{i \in V_2 } u_{i}^\gamma  \le (\frac{1}{2} + \sqrt{\e})^\gamma$ (similarly for $V_1$), which together with \eqref{41} imply
\begin{align} \label{651}
    \sum_{i \in V_1 } u_{i}^\gamma  \ge  \Big(\frac{1-\e}{2} \Big)^{2\gamma} 
  \Big(\frac{1}{2}+\sqrt{\e} \Big)^{-\gamma}
  > \Big( \frac{1}{2}-\sqrt{\e} \Big)^\gamma.
\end{align}

We claim that there exists a large constant $c = c(\gamma) >0$ such that {for any small enough $\ep>0$,} under the condition \eqref{eq: add cond}, 
\begin{align} \label{65}
    \max_{i\in V_1} u_i < \frac{1}{2}-c\sqrt{\e} \Rightarrow \sum_{i \in V_1 } u_{i}^{\gamma}\le \Big(\frac{1}{2}-2\sqrt{\e}\Big)^\gamma. 
\end{align}
Using inductively the fact $
    x^\gamma + y^\gamma \le (x+z)^\gamma + (y-z)^\gamma$ for any $\gamma>1$ and $x\ge y \ge z \ge 0,$ we deduce that 
if $\max_{i\in V_1} u_i < \frac{1}{2}-c\sqrt{\e}$ and $\sum_{i \in V_1 } u_{i} \le \frac{1}{2}+\sqrt{\e}$, then
\begin{align*}
    \sum_{i \in V_1 } u_{i}^{\gamma}\le \Big(\frac{1}{2}-c\sqrt{\e}\Big)^\gamma + ((c+1)\sqrt{\e})^\gamma.
\end{align*}
For a large $c>0$ {and a small  $\ep>0$}, this bound is at most $(\frac{1}{2}-2\sqrt{\e})^\gamma,$ establishing the claim \eqref{65}.

Therefore, by \eqref{651} and \eqref{65},   there exists $i_0\in V_1$ such that $u_{i_0} \ge \frac{1}{2}-c\sqrt{\e}.$ Similarly,  there exists $j_0\in V_2$ such that $u_{j_0} \ge \frac{1}{2}-c\sqrt{\e}.$
These vertices $i_0$ and $j_0$ are connected by an edge since otherwise
\begin{align*}
\Big(\frac{1-\e}{2}\Big)^{2\gamma}\le     \sum_{ (i,j)\in E } u_{i}^{\gamma} u_{j}^{\gamma}
    \le \Big(\sum_{i \in V_1 } u_{i}^{\gamma} \Big )  \Big(\sum_{j \in V_2 } u_{j}^{\gamma} \Big ) - u_{i_0}^\gamma u_{j_0}^\gamma \le \Big(\frac{1}{2}\Big)^{2\gamma} - \Big(\frac{1}{2}-c\sqrt{\e} \Big)^{2\gamma},
\end{align*}
yielding a contradiction.

\end{proof}

\section{Appendix}\label{sec6}

In this appendix, we verify Proposition \ref{prop: almost tree nbhd}, and then present basic properties about the largest eigenvalue of symmetric matrices. Finally, we state a crucial tail estimate on the sum of i.i.d.~Weibull distributions.

\subsection{Proof of Proposition \ref{prop: almost tree nbhd}}
We deduce Proposition \ref{prop: almost tree nbhd} by following  the argument in \cite[Proposition 4.1]{BHY19}.
\begin{proof}[Proof of Proposition \ref{prop: almost tree nbhd}]
For $i\in V(G)$, let $M_{i}$ be the excess in $B_{R_{n}}(i)$.  Then, since $M_i$
 is stochastically dominated by a binomial random variable $\text{Bin}(N,p)$ with $N:= d(d-1)^{R_n}$ and  $p := d(d-1)^{R_n-1}/n \ll 1$ 
 (see 
 \cite[Equation (2.4)]{LS10}),
    \begin{equation*}
        \P(M_{i} \ge w+1) \lesssim \binom{d(d-1)^{R_{n}}}{w+1}\cdot\paren{\frac{d(d-1)^{R_{n}-1}}{n}}^{w+1} \lesssim n^{-w-1}(d-1)^{2R_{n}(w+1)}.
    \end{equation*}
    By a union bound over all vertices $i\in V(G)$,
    \begin{equation}\label{eq: condition i violated}
    \P( \mathcal{E}_{1,1}^c) \le Cn^{-w}(d-1)^{2R_{n}(w+1)}.
    \end{equation}

    Next, let $M$ be the number of edges in $G$ which lie on cycles of length at most $2R_{n}$. Then, 
    \begin{equation*}
        \size{\parenn{i\in V(G):\text{$B_{R_n}(i)$ contains a cycle}}}
        \le 2(d-1)^{R_{n}}M
    \end{equation*}
    (see \cite[Equation (A.1)]{BHY19} for the explanation).
By \cite[Theorem 4]{MWW04} with $k\coloneqq 2R_{n}$ and $x\ge 2$,
    \begin{equation*}
        \P(M=40R_nx(d-1)^{2R_n}) \le (e^{5(x-1)}x^{-5x})^{(d-1)^{2R_n}} \le \exp( -c (d-1)^{2R_{n}} )
    \end{equation*}
    for some universal constant $c>0$.  By a union bound,  
    \begin{equation*}
        \P(M \ge 80R_{n}(d-1)^{2R_n} )
        \le dn \exp( -c (d-1)^{2R_{n}} ).
    \end{equation*}
 Hence, for sufficiently large $n$, with probability at least $1-dn \exp( -c (d-1)^{2R_{n}} )$,
    \begin{equation*}
        \size{\parenn{i\in V(G):\text{$B_{R_n}(i)$ contains a cycle}}}
        \le 2(d-1)^{R_{n}}\cdot 80R_{n}(d-1)^{2R_{n}}\le (d-1)^{4R_n}.
    \end{equation*} 
This implies
    \begin{equation}\label{eq: condition ii violated}
    \P( \mathcal{E}_{1,2}^c) \le  dn \exp( -c (d-1)^{2R_{n}} ) .
    \end{equation} 
If $\log \log n \ll R_n \ll \log n,$ then for any constant $\e>0,$ bounds \eqref{eq: condition i violated} and \eqref{eq: condition ii violated} are at most $n^{-w+\e}$ for sufficiently large $n$.
\end{proof}

\subsection{Largest eigenvalue of symmetric matrices}
First, we state the  variational formula for the largest eigenvalue of symmetric matrices.
\begin{lemma} \label{variation}
    Let $X=(X_{ij})_{i,j\in [n]}$ be any symmetric matrix. Then,
    \begin{align} \label{var}
        \lambda_1(X)  = \sup_{\norm {\textup{\textbf{v}}}_2 = 1}  \sum_{i,j \in [n]}  X_{ij}v_iv_j.
    \end{align}
\end{lemma}
Next, we  state a trivial lower bound for the largest eigenvalue of symmetric matrices, obtained as an immediate consequence of Lemma \ref{variation}.
\begin{lemma} \label{variation2}
    Let $X=(X_{ij})_{i,j\in [n]}$ be any symmetric matrix {with zero diagonal.} Then,
    \begin{align}
        \lambda_1(X) \geq  \max_{i,j\in [n]} |X_{ij}|.
    \end{align}
\end{lemma}

\subsection{Sum of i.i.d.~Weibull  distributions}
Finally,  we state a lemma regarding the tail of i.i.d.~sum of  Weibull distributions conditioned to be large in absolute value.

\begin{lemma}\label{lem: conditioned alpha power sum asymp}
For any  $b>1$, let  $\{Y_i\}_{i\ge 1}$ be i.i.d.~Weibull random variables with a shape parameter $\alpha>0$, conditioned to be greater than $b^{1/\alpha}$ in absolute value. Then, there exists a constant $C>0$ (depending only on $C_1$ and $C_2$ from Definition \ref{def: weights}) such that for any  $m\in \N$ and $L>m$,
	\begin{equation*}
	\P(|Y_1|^\alpha+\cdots+|Y_m|^\alpha\ge L) \le
	\paren{\frac{CL}{m}}^{m}e^{-L+m+ mb} \le  e^{-L+m+ mb + m \log (CL)}.
	\end{equation*}
\end{lemma}

\begin{proof}
By  the  exponential Chebyshev bound, for any $s>0$,
\begin{align} \label{500}
\mathbb{P} \left (|Y_1|^\alpha+\cdots+|Y_m|^\alpha \geq    L \right )   \leq e^{-  sL }  \mathbb{E}\left [ e^{s|Y_1|^\alpha} \right ]^m.
\end{align}
By a tail bound  of the Weibull distribution (see Definition \ref{def: weights}), there exists a constant $C \geq 1$  such that   for $x> b^{1/\alpha} $,
\begin{align*}
\mathbb{P} \left ( |Y_1 | \ge x \right ) \leq   C e^{ b }    e^{- x^\alpha } ,
\end{align*}
and for $0\leq x\leq b^{1/\alpha}$, $\mathbb{P}(|Y_1 | \ge  x)  = 1$.
Hence,  for any {$0<s< 1$,}
\begin{align*}
\mathbb{E} \left [ e^{s|Y_1|^\alpha} \right ] &= 1  +  \int_0^\infty  e^{sx^\alpha} s \alpha x^{\alpha-1 } \mathbb{P} \left (|Y_1|  \ge  x \right ) dx \\
&\le   1+ \int_0^{  b^{1/\alpha}  } e^{sx^\alpha} s \alpha x^{\alpha-1 }  dx +  C \int_{ b^{1/\alpha}   }^\infty  e^{sx^\alpha} s \alpha x^{\alpha-1 } \cdot  e^{ b }    e^{- x^\alpha } dx \\
&\leq      e^{ s b }      +   C e^{ b } \frac{s}{ 1-s} e^{ (s-1)b } =   \Big(1+ \frac{Cs}{ 1-s}\Big) e^{ s b } .
\end{align*} 
Applying this to \eqref{500}, for any $0<s< 1$,
\begin{align*}
\mathbb{P} \left ( |Y_1|^\alpha+\cdots+|Y_m|^\alpha \geq    L \right ) \leq    e^{-sL}   \Big ( 1+ \frac{Cs}{ 1-s}  \Big)^m        e^{ s m b } .
\end{align*}
Setting $s: =    1- \frac{m}{  L   } \in (0,1) $ (recall that $L>m$), using the fact $C\ge 1,$
\begin{align*}
\mathbb{P} \left ( |Y_1|^\alpha+\cdots+|Y_m|^\alpha \geq    L \right ) &\leq  e^{-  L+m}  \left (  1+ \frac{C(L-m)}{m}  \right )^m    e^{ m b }\leq     e^{-  L+m}   \left ( \frac{CL}{m} \right )^m    e^{m b},
\end{align*}
which concludes the proof.

\end{proof}


\end{document}